\documentclass[11pt,a4paper,reqno]{amsart}
\usepackage{a4wide}
\usepackage[utf8x]{inputenc}
\usepackage{ucs}
\usepackage{amsmath,amsthm}
\usepackage{amsfonts}
\usepackage{amssymb}
\usepackage{paralist}
\usepackage{graphicx}

\title{Stability index for chaotically driven concave maps}

\author{Gerhard Keller}
\address{Department Mathematik, Universität Erlangen-Nürnberg, Cauerstr.~11, 91058 Erlangen, Germany}
\email{keller@mi.uni-erlangen.de}
\thanks{This work is funded by DFG grant Ke 514/8-1. I am indebted to A. Otani (Erlangen) for many helpful remarks on this paper and for producing Figure 1, and I thank R. Ramaswamy and his group (University of Hyderabad, India) for their  hospitality during a visit in March 2012, where a huge part
of this research was done. }
\date{\today}

\newcommand{\R}{\mathbb{R}}
\newcommand{\N}{\mathbb{N}}
\newcommand{\hS}{{\hat S}}
\newcommand{\T}{{\mathbb{T}^1}}

\newcommand{\LL}{\mathcal{L}}
\renewcommand{\AA}{\mathcal{A}}
\newcommand{\ppl}{\varphi_{\infty}}
\newcommand{\hppl}{\hat{\varphi}_{\infty}}
\newcommand{\nn}{\nonumber}
\newcommand{\mac}{{\mu_{\rm ac}}}
\renewcommand{\vartheta}{{v}}
\newcommand{\fs}{\,.}
\newcommand{\co}{\,,}

\newtheorem{corollary}{Corollary}

\newtheorem{lemma}{Lemma}
\newtheorem{theorem}{Theorem}
\newtheorem{proposition}{Proposition}
\theoremstyle{definition}
\newtheorem{example}{Example}

\newtheorem{remark}{Remark}
\newtheorem{hypothesis}{Hypothesis}

\numberwithin{equation}{section}

\begin{document}
\begin{abstract}
We study skew product systems driven by a hyperbolic base map $\hat S:\Theta\to\Theta$ (e.g. a baker map or an Anosov surface diffeomorphism) and with simple concave fibre maps on $\R_+$ like 
$x\mapsto  \hat g(\theta)\arctan(x)$ where $\theta\in\Theta$ is a parameter driven by the base map.
The fibre-wise attractor is the graph of an upper semicontinuous function $\theta\mapsto\hppl(\theta)\in\R_+$. 
For many choices of $\hat g$, $\hppl$ has a residual set of zeros but $\hppl>0$ $\mu_{\text{SRB}}$-a.s. where $\mu_{\text{SRB}}$ is the Sinai-Ruelle-Bowen measure of $\hat S^{-1}$. 

In such situations we evaluate the stability index of the global attractor of the system, which is the subgraph 
$\{(\theta,x)\in\Theta\times\R_+: 0\leqslant x\leqslant\hppl(\theta)\}$ of $\hppl$, at all regular points $(\theta,0)$ in terms of the local exponents
$\hat{\Gamma}(\theta):=\lim_{n\to\infty}\frac{1}{n}\log\hat{g}_n(\theta)$ and $\hat{\Lambda}(\theta):=\lim_{n\to\infty}\frac{1}{n}\log|D_u\hS^{-n}(\theta)|$
and of the positive zero ${s_*}$ of a certain thermodynamic pressure function associated with $\hat S$ and $\hat g$. (In queuing theory, an analogon of ${s_*}$ is known as Loyne's exponent \cite{Loynes1962}.)

The stability index was introduced by Podvigina and Ashwin \cite{Podvigina-Ashwin2011} to quantify the local scaling of basins of attraction.

\end{abstract}

\maketitle

\section{Introduction}

\subsection{Motivation}
Consider a monotone concave map $h$ that maps some interval $[0,a]$ into itself with $h(0)=0$ and $h'(0)=1$. The family $h_r(x)=rh(x)$ with $0\leqslant r\leqslant h(a)^{-1}$ has a very simple bifurcation scenario: for $r\leqslant1$, the point $0$ is a globally attracting fix point, that looses its stability at $r=1$ and gives birth to a new stable fixed point $x_s>0$ which attracts all points except the fixed point $0$.

If the bifurcation parameter $r$ is not fixed but is driven by some ergodic dynamics, the scenario becomes a bit more complex. Quasiperiodic drives may lead to the creation of strange non-chaotic attractors (SNA) as the result of the loss of stability of a stable non-autonomous fixed point, a phenomenon that attracted much attention both in the physics and the mathematics literature, see e.g. the references collected in \cite{Feudel-book,Jager2009}.
More recently, also systems with chaotic drives were studied - mostly in the physics literature where they are used as simple examples to study generalized synchronisation, see e.g.~\cite{Singh2008}. Due to the presence of many different normal Lyapunov exponents associated to different invariant measures of the chaotic driving system, 
the loss of stability of the globally attracting non-autonomous fixed point at~$0$ and the creation of an attracting non-autonomous fixed point which is everywhere strictly positive is a complicated process that happens while the parameter varies in a nontrivial interval \cite{Singh2008}.
The goal of this paper is to describe some quantitative features of this process in simple model situations.

\subsection{The class of systems}
We study skew product systems where the driving system is
a bijective bi-measurable map $\hS:\Theta\to\Theta$ on a measurable space $(\Theta,\AA)$ that has good hyperbolicity properties to be specified below. The fibre maps from an interval $I:=[0,a]$ into itself are of the form $x\mapsto\hat g(\theta)h(x)$ where $\hat g:\Theta\to(0,\infty)$
and $h:I\to\R_+$ is a strictly increasing, concave $C^{1+}$-function with $h(0)=0$ and $h'(0)=1$.\footnote{Here and in the sequel $C^{1+}$ means "$C^1$ with Hölder continuous derivative" without specifying the Hölder exponent.}
Let $\Omega=\Theta\times I$. Then the driven system is described by
\begin{equation}
F:\Omega\to\Omega,\quad
F(\theta,x)=(\hS\theta,\hat g(\theta)h(x))\fs
\end{equation}
Denote by $F^n_\theta:I\to I$ the fibre map of the iterated map $F^n$, i.e. 
$F^n_\theta(x)$ is the second component of $F^n(\theta,x)$.

The global pullback attractor of this system is the set
\begin{equation}
\{(\theta,x)\in\Omega: 0\leqslant x\leqslant\hppl(\theta)\}
\end{equation}
where $\hppl:\Theta\to I$ is the \emph{maximal invariant graph} (with the slight abuse of terminology that we do not distinguish between the function and its graph). It is defined for all $\theta\in\Theta$ by
\begin{equation}\label{eq:hat_phi_def}
\hppl(\theta)=\lim_{n\to\infty}\hat{\varphi}_n(\theta),
\text{ where\quad}
\hat{\varphi}_n(\theta):=F^n_{\hS^{-n}\theta}(a)\fs
\end{equation}
The limit exists and is measurable, because
$\hat\varphi_{n+1}(\theta)=F^n_{\hS^{-n}\theta}(F_{\hS^{-(n+1)}\theta}(a))
\leqslant F^n_{\hS^{-n}\theta}(a){=\hat\varphi_n(\theta)}$ in view of the monotonicity of the fibre maps. If $\Theta$ is a topological space and if all $\hat{g}\circ\hS^{-n}$ are continuous, then also all $\hat{\varphi}_n$ are continuous so that $\hppl$ is upper semicontinuous.

In order to obtain some quantitative, dimension-like information about $\hppl$, we need some additional uniformly hyperbolic or expanding structure for the system. The following assumptions are a compromise between the goal to cover a number of different examples and to keep technicalities at a moderate level. 
\begin{hypothesis}\label{hypo:1}
There is a piecewise expanding and piecewise $C^{1+}$ mixing Markov map 
$S:\T\to\T$ with finitely many branches which is a factor of $\hS^{-1}$, i.e.
\begin{equation}
S\circ\Pi=\Pi\circ\hS^{-1}\text{\quad for some measurable\quad}\Pi:\Theta\to\T\fs
\end{equation}
It is a well known fact that $S$ has a unique invariant probability measure $\mac$ absolutely continuous w.r.t. Lebesgue measure $m$ on $\T$.
\end{hypothesis}
\begin{remark}
One can also admit countable Markov maps with finite range structure, and a careful look at the proofs reveals possibilities to weaken the assumption on $S$ even further.
\end{remark}
\begin{hypothesis}\label{hypo:2}
The multiplier function $\hat g$ depends only on $\Pi\theta$, i.e.
\begin{equation}\label{eq:factor-g}
\hat g(\theta)=g(\Pi\theta)
\end{equation}
for a suitable function $g:\T\to(0,\infty)$. (How to deal with more general multiplier functions when $\hS$ is (piecewise) hyperbolic, is explained in Remark~\ref{remark:general-mult}.)
Let $g_n=\prod_{i=1}^ng\circ S^i$, and denote
by $\mathcal{U}_n(\vartheta)$ the family of all interval neighbourhoods $U$ of $\vartheta\in\T$ such that $S^n_{|U}:U\to S^nU$ is a diffeomorphism.
We assume that the family of all
${g_n}_{|U}$ with $n\geqslant1$, $\vartheta\in\T$ and $U\in\mathcal{U}_n(\vartheta)$ has uniformly bounded   distortion in the following sense:
There is a constant $D>0$ such that for all $n>0$, all $\vartheta\in\T$, all $U\in\mathcal{U}_n(\vartheta)$ and all ${\tilde\vartheta}\in U$ 
\begin{equation}\label{eq:bounded-osc}
D^{-1}\leqslant\left|\frac{g_n({\tilde\vartheta})}{g_n(\vartheta)}\right|\leqslant D\fs
\end{equation}
\end{hypothesis}

\begin{remark}\label{remark:distortion}
If $\log g$ is Hölder continuous on each monotonicity interval of $S$, assumption (\ref{eq:bounded-osc}) is a simple classical consequence of the uniform expansion of $S$. Similarly we have
(enlarging $D$, if necessary)
\begin{equation}
D^{-1}\leqslant\left|\frac{(S^n)'({\tilde\vartheta})}{(S^n)'(\vartheta)}\right|\leqslant D\fs
\end{equation}
\end{remark}

\begin{remark}
The variable $\theta$ enters the definition of the approximating functions $\hat{\varphi}_n$ only via the values $\hat g(\hS^{-k}\theta)=g(S^k(\Pi\theta))$, $k=1,\dots,n$. Therefore the graph $\hppl(\theta)$ depends on $\theta$ only via 
$\Pi\theta$ so that there is a measurable function $\ppl:\T\to I$ such that $\hppl(\theta)=\ppl(\Pi\theta)$. The geometric properties of this function are what we are basically interested in. 
Corresponding properties of the function $\hppl$ will follow as corollaries.
\end{remark}

The following is a well known consequence of the semi-uniform ergodic theorem \cite{Sturman2000} and of the uniform concavity of the fibre maps: $\ppl(\vartheta)=0$ for all $\vartheta\in\T$ if
$\int_\T \log g\,d\mu<0$ for all $S$-invariant probability measures $\mu$, and $\ppl$ is strictly positive if $\int_\T \log g\,d\mu>0$ for all such~$\mu$. The most interesting situation occurs under the following hypothesis:
\begin{hypothesis}
\label{hypo:3}
There is an $S$-invariant probability measures $\mu_-$ such that
\begin{equation}\label{eq:ass-minus-plus}
\int\log g\,d\mu_-<0<\int\log g\,d\mac\fs
\end{equation}
Note that under this assumption $\log g$ is not cohomologous to a constant and that it is easy to prove (see \cite{Keller1996,KeOt2012}) 
that $\ppl(\vartheta)>0$ for $\mac$-a.e. $\vartheta$.
\end{hypothesis}

\begin{example}[\textbf{Baker transformations}]
\label{ex:baker}
Let $\Theta=[0,1)^2$ and let $\hS:\Theta\to\Theta$ be a baker transformation
\begin{equation}
\hS(u,\vartheta)=
\begin{cases}
\left(s^{-1}u,s\vartheta\right)&\text{if }u<s\\
\left((1-s)^{-1}(u-s),s+(1-s)\vartheta\right)&\text{if }u\geqslant s\fs
\end{cases}
\end{equation}
With $\Pi(u,\vartheta)=\vartheta$ and with $S(\vartheta)=s^{-1}\vartheta$ for $\vartheta<s$ and $S(\vartheta)=(1-s)^{-1}(\vartheta-s)$ if $\vartheta\geqslant s$ this fits the above setting. Figure~\ref{fig:baker} shows plots of the invariant graph $\ppl(\vartheta)$ when $s=0.45$,
$h(x)=\arctan(x)$ and the multiplier function $g:\T\to(0,\infty)$ is $g(\vartheta)=r\cdot(1+\epsilon+\cos(2\pi \vartheta))$ with $\epsilon=0.01$. 
Observe that in this example all $\hat{g}\circ\hS^{-n}$ are continuous when interpreted as defined on the circle $\T$ so that $\hppl$ and $\ppl$ are upper semicontinuous.
Our main results shed some light on the structure of $\ppl$ close to the base line, i.e. when these values are small. 

In this example, the $S$-invariant measure $\delta_0$ maximizes $\int\log g\,d\mu$ (the value is $\log (r\cdot2.01)$), and the equidistribution on the period-3 orbit
$[0.10255, 0.22788, 0.50640]$ apparently minimizes this quantity (the value is $\log (r\cdot0.28216)$). The corresponding value for Lebesgue measure $\mu=m$ is $\log (r\cdot0.57589)$. So assumption (\ref{eq:ass-minus-plus}) is satisfied
for parameters $r\in[0.57589^{-1},0.28216^{-1}]=[1.7364,3.5441]$, and the parameters used in Figure~\ref{fig:baker} are in this range.
\begin{figure}
\label{fig:baker}
\begin{center}
\includegraphics[height=50mm,width=55mm]{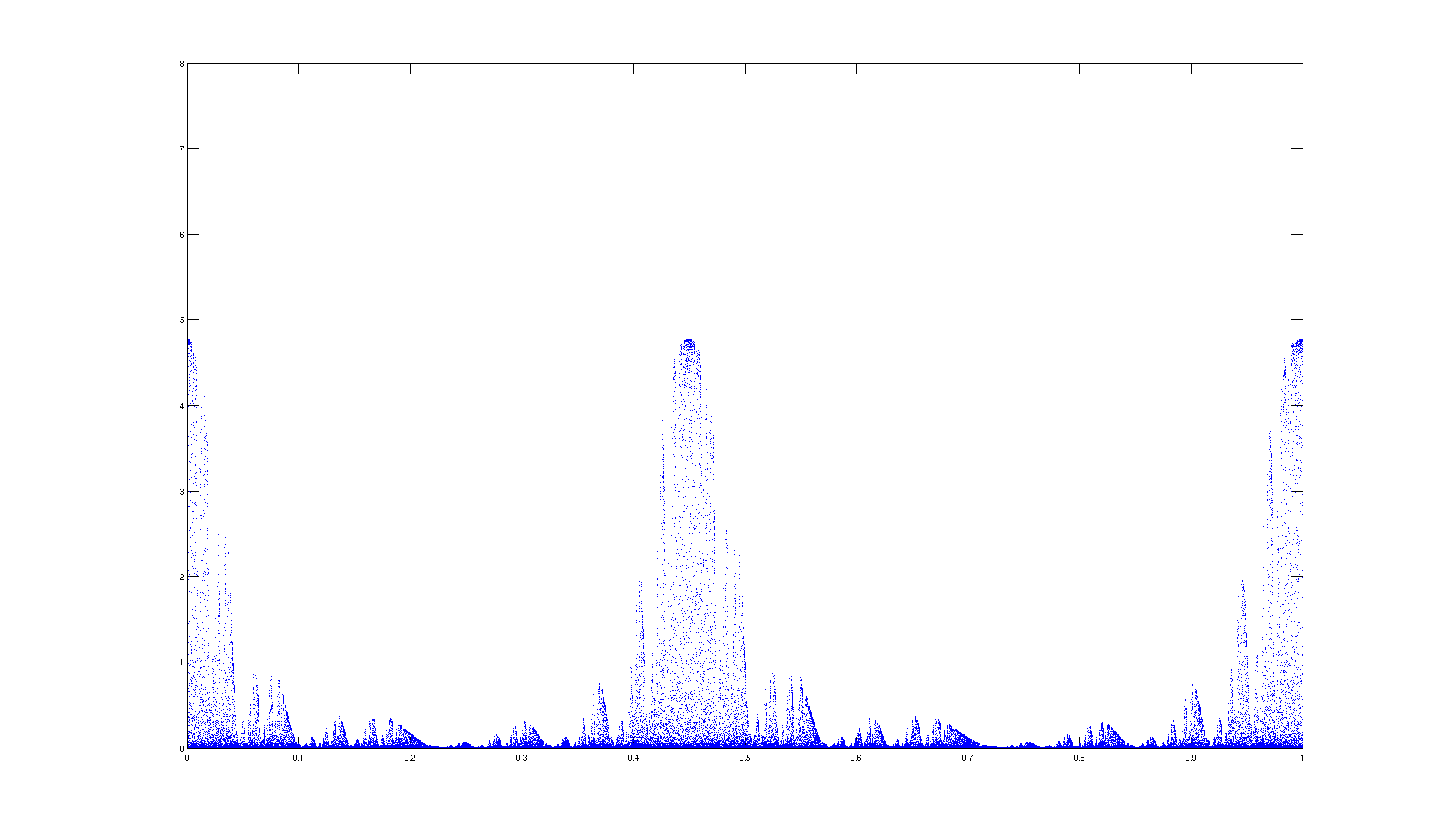}\hspace{-0.6cm}
\includegraphics[height=50mm,width=55mm]{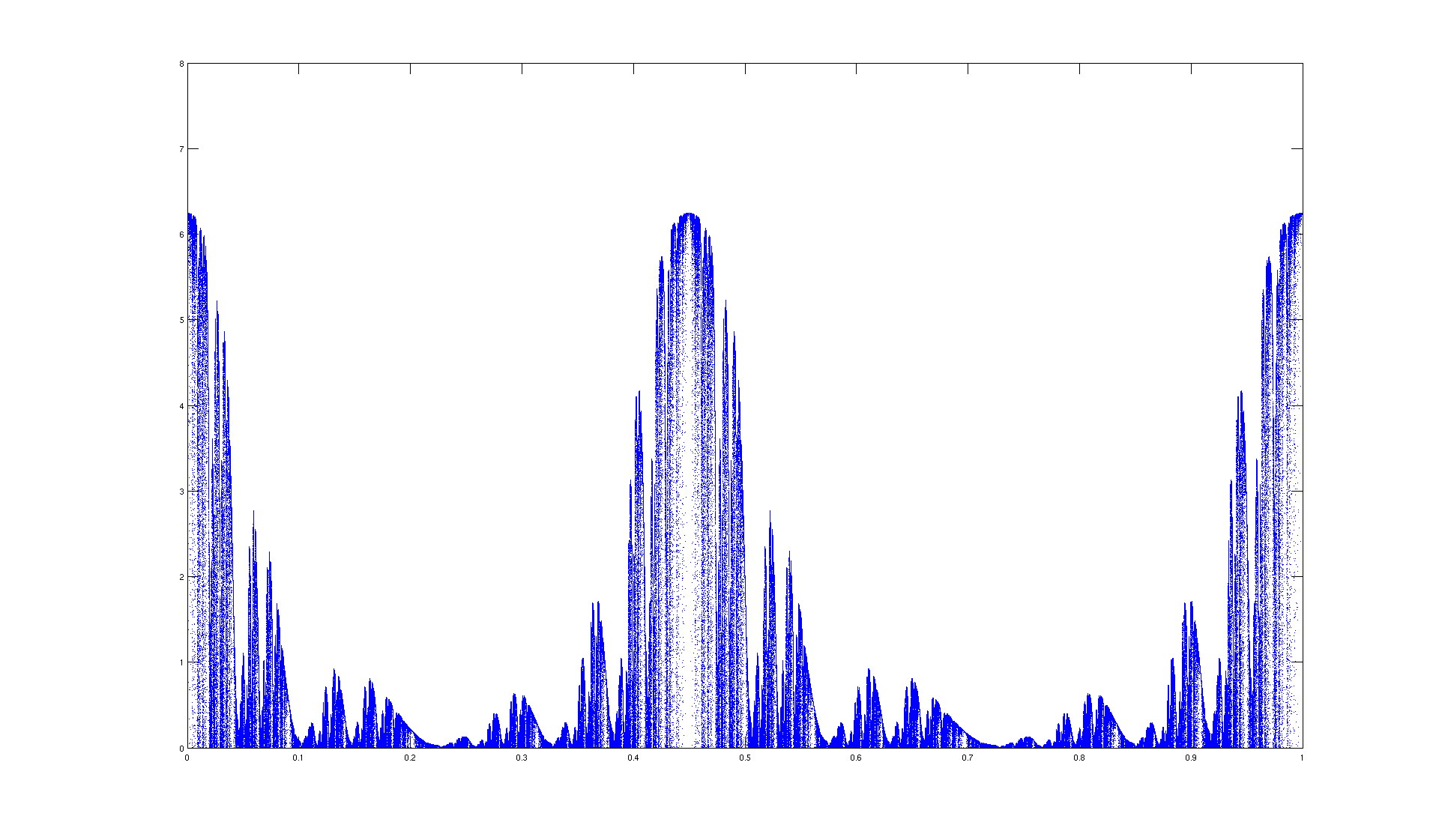}\hspace{-0.6cm}
\includegraphics[height=50mm,width=55mm]{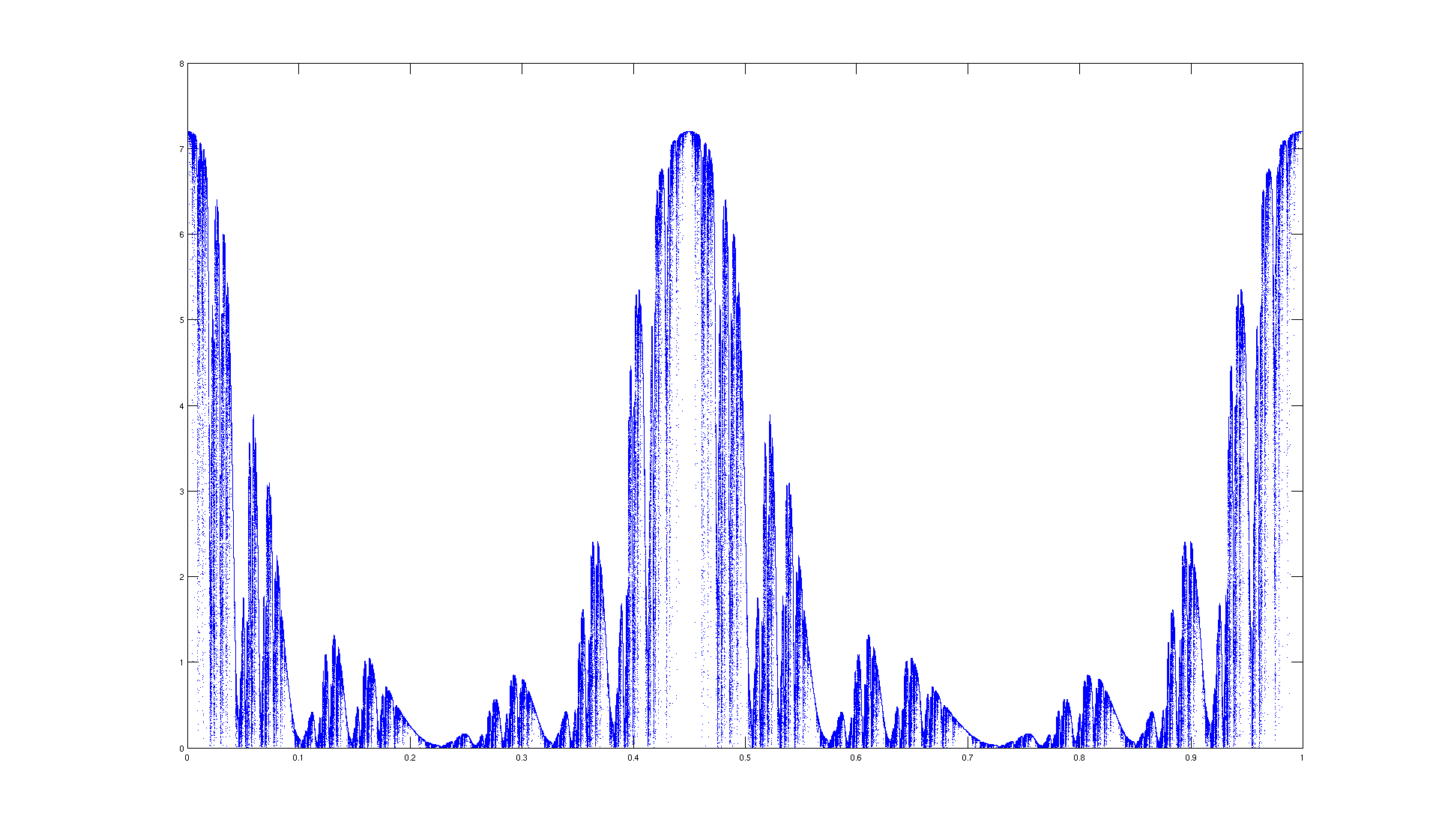}
\end{center}
\caption{The graph $\ppl(\vartheta)$ for the baker map from Example \ref{ex:baker}. The parameters are (from left to right) $r=1.74$, $r=2.2$, $r=2.5$.}
\end{figure}
\end{example}

\begin{remark}\label{remark:general-mult}
Baker transformations are particularly simple examples where the sets $\Pi^{-1}(\vartheta)$ are uniformly stable fibres for the action of $\hS^{-1}$ on $\Theta$. In such situations one can also deal with multiplier functions $\hat{g}(\theta)$ that do not only depend on $\Pi\theta$ as required in Hypothesis~\ref{hypo:2}. 
Under suitable assumptions, a classical construction which goes back to works of Sinai and of Bowen yields functions $\hat{b}:\Theta\to\R$ and $g:\T\to(0,\infty)$ such that
\begin{equation}\label{eq:comology_for_g}
\log\hat{g}(\theta)=\log g(\Pi\theta)+\hat{b}(\theta)-\hat{b}(\hS^{-1}\theta)
\fs
\end{equation}
More precisely, we assume:
\begin{enumerate}[i)]
\item $\log\hat g:\Theta\to\R$ is Hölder continuous. (Hölder continuity on each set $\Pi^{-1}J$ where $J$ is a monotonicity interval of $S$ suffices.)
\item There is an injection $\varsigma:\T\to\Theta$ which is Hölder continuous on monotonicity intervals of $S$, which satisfies  $\Pi\circ\varsigma=\operatorname{id}_\T$,
and which is such that each $\theta\in\Theta$ belongs to the stable fibre of $\varsigma\Pi\theta$ in the sense that
\begin{equation}
\exists C>0\ \exists r\in(0,1)\ \forall \theta\in\Theta\ \forall n>0:\ 
{d}(\hS^{-n}\theta,\hS^{-n}(\varsigma\Pi\theta))\leqslant
C\,r^n\fs
\end{equation}
\end{enumerate}
Following \cite[Lemma 1.6]{Bowen}, define
\begin{equation}
\hat{b}(\theta)=\sum_{n=0}^\infty\left(\log\hat{g}(\hS^{-n}\theta)-\log\hat{g}(\hS^{-n}\varsigma\Pi\theta)\right)\fs
\end{equation}
As $\log\hat{g}$ is Hölder continuous, $\|\hat{b}\|_\infty:=\sup_{\theta\in\Theta}|\hat{b}(\theta)|<\infty$, and 
\begin{equation*}
\hat{b}(\theta)-\hat{b}(\hS^{-1}\theta)
=
\log\hat{g}(\theta)-\left[\log\hat{g}(\varsigma\Pi\theta)
+\sum_{n=1}^\infty\left(\log\hat{g}(\hS^{-n}\varsigma\Pi\theta)-\log\hat{g}(\hS^{-n+1}\varsigma\Pi\hS^{-1}\theta)\right)\right]
\fs
\end{equation*}
The term in brackets depends only on $\varsigma\Pi\theta$, and we denote it by $\log g(\Pi\theta)$. Then
\begin{equation}\label{eq:log_g-cohom}
\hat{b}(\theta)-\hat{b}(\hS^{-1}\theta)
=
\log\hat{g}(\theta)-\log g(\Pi\theta)\co
\end{equation}
and one can show that $\hat{b}$ and $\log\hat{g}$ are Hölder continuous \cite[Lemma 1.6]{Bowen}. In particular, the distortion bounds of Hypothesis~\ref{hypo:2} are satisfied.

Denote now by $\hppl$ the invariant graph of the system with multiplier $\hat{g}$, and by $\ppl\circ\Pi$ the invariant graph of the system with multiplier $g\circ\Pi$. 
We prove the following proposition in section~\ref{sec:anosov-proofs}.
\end{remark}

\begin{proposition}\label{prop:anosov-lemma}
For each $\theta\in\Theta$, 
$\hppl(\theta)>0$ if and only if 
$\varphi_\infty(\Pi\theta)>0$, and if this is the case, then
\begin{equation}\label{eq:anosov-lemma}
|\log\hppl(\theta)-\log\varphi_\infty(\Pi\theta)|
\leqslant
\log\frac{a}{h(a)}+2\|\hat{b}\|_\infty\fs
\end{equation}
\end{proposition}

\begin{example}[\textbf{Anosov surface diffeomorphism}]
\label{ex:anosov}
Let $\Theta=\mathbb{T}^2$ and let $\hS:\mathbb{T}^2\to\mathbb{T}^2$ be a $C^2$ Anosov diffeomorphism. 
It has a Markov partition $\{R_1,\dots,R_p\}$ \cite{Robinson1995}. 
As indicated in the proof of Lemma 3 in \cite{Pollicott2003} (see also section~\ref{subsec:Anosov-Markov})
one can construct a $C^{1+}$ expanding Markov interval map $S:\T\to\T$ that is a factor of $\hS^{-1}$, i.e.
$S\circ\Pi=\Pi\circ\hS^{-1}$ with the projection $\Pi:\mathbb{T}^2\to\T$ 
and the injection $\varsigma:\T\to\mathbb{T}^2$ defined in section~\ref{subsec:Anosov-Markov}.
If $\hat g:\mathbb{T}^2\to(0,\infty)$ is a Hölder function, then there are functions $g:\T\to(0,\infty)$ and 
$\hat{b}:\mathbb{T}^2\to\R$
such that $\log\hat{g}=\log g\circ\Pi+\hat{b}-\hat{b}\circ\hS^{-1}$ and $\log {g}$ is Hölder continuous on every monotonicity interval of $S$, compare Remark~\ref{remark:general-mult}. 

Denote by $\hat\mu_{\text{SRB}}^-$ the SRB-measure of $\hS^{-1}$. It projects to a $S$-invariant measure $\mu_{\text{SRB}}^-$ on $\T$.
As $\hat\mu_{\text{SRB}}^-$ is absolutely continuous on unstable fibres of $\hS^{-1}$, the measure $\mu_{\text{SRB}}^-\circ\Pi^{-1}$ is absolutely continuous w.r.t. Lebesgue measure on $\T$, so that it coincides with the unique absolutely continuous invariant measure $\mac$ of $S$ 
from Hypothesis~\ref{hypo:1}. 
Using the explicit representation for the Jacobian $D\Pi$ of the holonomy along stable fibres of $\hS^{-1}$ (in this case: the absolute value of the derivative of the holonomy), it is not hard to prove that
\begin{equation}\label{eq:log_DS-cohom}
\log|D_u\hS^{-1}(\theta)|=\log |S'(\Pi\theta)|+\log D\Pi(\theta)-\log D\Pi(\hS^{-1}\theta)\fs
\end{equation}
For completeness the proof is provided in section~\ref{sec:anosov-proofs}. Here $D_u$ denotes the derivative in the unstable direction of $\hS^{-1}$.
Proposition~\ref{prop:anosov-lemma} applies in this situation so that
the graphs of $\hppl$ and of $\varphi_\infty\circ\Pi$ can again be compared as in (\ref{eq:anosov-lemma}).
\end{example}

\section{Main results}
Throughout we assume that Hypotheses 1 - 3 are satisfied. 

\subsection{Global scaling properties}
\label{subsec:global-sc}

A global characteristic of the invariant graph $\ppl:\T\to[0,\infty)$	is the distribution of its values - in particular of values close to zero -  under Lebesgue measure $m$. Recall that $\ppl(\vartheta)>0$ for $m$-a.e.~$\vartheta\in\T$ by Hypothesis~\ref{hypo:3}.

For $s\in\R$ denote by $\LL_s$ the transfer operator 
\begin{equation}\label{eq:PF}
\LL_s:L^1_m\to L^1_m,\; \LL_sf(\vartheta)=\sum_{\tilde{\vartheta}\in S^{-1}\vartheta}\frac{f(\tilde{\vartheta})}{|S'(\tilde{\vartheta})|}e^{-s \log g(\tilde{\vartheta})}\co
\end{equation}
and let $\rho(\LL_s)$ be its spectral radius. Define $\psi(s)=\log\rho(\LL_s)$, and
observe that $\psi(s)$ is the topological pressure of the potential $-\log|S'|-s\log g$ under the dynamics of $S$
\cite{ParryPollicott}. \footnote{To be more precise, it is the pressure of the topological Markov chain that encodes $S$.}

The operator $\LL_0$ is the usual Perron-Frobenius operator of $S$, so $\psi(0)=0$ and
$\psi'(0)=-\int\log g\,d\mac<0$, see e.g. \cite{ParryPollicott}. From the assumption
in Hypothesis~\ref{hypo:3} that there is also a measure $\mu_-$ with 
$-\int\log g\,d\mu_->0$, it follows that $\psi(s)\to\infty$ as $s\to\infty$. Because of its convexity, $\psi(s)$ has therefore a unique further zero ${s_*}>0$. This number characterizes the distribution of "small values" of $\ppl$ in the sense of the following theorem.

\begin{theorem}\label{theo:tail}
\begin{equation}\label{eq:tail-scaling}
\lim_{x\to\infty}\frac{1}{x}\log m\{\vartheta\in\T: \log\ppl(\vartheta)<-x\}=-{s_*}\fs
\end{equation}
Replacing $-x$ by $\log \epsilon$, this can be reformulated as
\begin{equation}
\lim_{\epsilon\to0}\frac{\log m\{\ppl<\epsilon\}}{\log\epsilon}={s_*}\fs
\end{equation}
\end{theorem}

For the local analysis of $\ppl$ (see section~\ref{subsec:local-scaling}) we also need a modification of this last identity. Define
\begin{equation}
\Xi_\epsilon:=\frac{1}{\epsilon}\int_\T\min\{\ppl(t),\epsilon\}\,dt
\quad(\epsilon>0)\co
\end{equation}
so that
\begin{equation}
1-\Xi_\epsilon=\frac{1}{\epsilon}\int_\T(\epsilon-\ppl(t))^+\,dt\fs
\end{equation}

\begin{theorem}\label{theo:global-scaling}
\begin{equation}\label{eq:global-scaling}
\lim_{\epsilon\to0}\frac{\log \Xi_\epsilon}{\log\epsilon}=0
\text{\quad and \quad}
\lim_{\epsilon\to0}\frac{\log(1- \Xi_\epsilon)}{\log\epsilon}={s_*}\fs
\end{equation}
\end{theorem}
\noindent The proofs of (slight generalisations of) these two theorems are provided in 
section~\ref{sec:proof-global-dist}.

\subsection{Local scaling properties}\label{subsec:local-scaling}

As in \cite{Podvigina-Ashwin2011} we define a local stability index $\sigma(\vartheta)$ of the invariant graph $\ppl$ in the following way:
\begin{equation}
\sigma(\vartheta):=\sigma_+(\vartheta)-\sigma_-(\vartheta)
\end{equation}
where
\begin{equation}\label{eq:local-scaling}
\sigma_-(\vartheta):=\lim_{\epsilon\to0}\frac{\log\Sigma_\epsilon(\vartheta)}{\log\epsilon}\quad\text{and}\quad
\sigma_+(\vartheta):=\lim_{\epsilon\to0}\frac{\log(1-\Sigma_\epsilon(\vartheta))}{\log\epsilon}
\end{equation}
with
\begin{equation}
\Sigma_\epsilon(\vartheta):=
\frac1{\epsilon\cdot|U_\epsilon(\vartheta)|}\int_{U_\epsilon(\vartheta)}
\min\{\ppl(t),\epsilon\}\,dt
\end{equation}
and
\begin{equation}
1-\Sigma_\epsilon(\vartheta)=
\frac{1}{\epsilon\cdot|U_\epsilon(\vartheta)|}
\int_{U_\epsilon(\vartheta)}(\epsilon-\ppl(t))^+\,dt\fs
\end{equation}
The $U_\epsilon(\vartheta):=(v-\epsilon,v+\epsilon)$ are symmetric interval neighbourhoods of $\vartheta$ of size $2\epsilon$.

Of course, the limits in (\ref{eq:local-scaling}) need not exist \emph{a priori}, but sufficient conditions for their existence are formulated in Theorem~\ref{theo:main}. If $\sigma_+(v)$ and $\sigma_-(v)$ both exist, they are non-negative and at most one of them can be strictly positive.

For $\theta\in\Theta$ we define $\hat{{\sigma}}_\pm(\theta)={\sigma}_\pm(\Pi\theta)$.

\begin{proposition}
$\hat{{\sigma}}_\pm(\hS\theta)=\hat{{\sigma}}_\pm(\theta)$ for all $\theta\in\Theta$.
\end{proposition}
\noindent This is essentially Theorem 2.2 of \cite{Podvigina-Ashwin2011}. Observe just that the proof of that theorem applies to any forward and backward invariant set.

\begin{corollary}
For each ergodic $\hS$-invariant measure $\hat\mu$ the function $\hat{{\sigma}}_\pm$ is $\hat\mu$-a.s. constant.
\end{corollary}


Recall from Hypotheis~\ref{hypo:2} that
$\mathcal{U}_n(\vartheta)$ denotes the family of all interval neighbourhoods $U$ of $\vartheta\in\T$ such that $S^n_{|U}:U\to S^nU$ is a diffeomorphism.
The following theorem is proved in section \ref{sec:proofs-local-scaling}.

\begin{theorem}\label{theo:main}
Let $\vartheta\in\T$ be \emph{regular} in the sense that
\begin{equation}\label{eq:Lambda-Gamma-ass}
\Gamma(\vartheta):=\lim_{n\to\infty}\frac{1}{n}\log g_n(\vartheta)\text{\quad and \quad}
\Lambda(\vartheta):=\lim_{n\to\infty}\frac{1}{n}\log|(S^n)'(\vartheta)|
\end{equation}
exist and that 
\begin{equation}\label{eq:extra-ass}
\parbox{0.85\linewidth}{there are sequences $n_1<n_2<\dots$ of integers and $U_{\epsilon_1}\supseteq U_{\epsilon_2}\supseteq\dots$ of \emph{symmetric} interval
neighbourhoods of $\vartheta$ with $U_{\epsilon_k}\in\mathcal{U}_{n_k}(\vartheta)$ such that 
\[
\lim_{k\to\infty}\frac{n_k}{n_{k+1}}=1\quad\text{and}\quad
\Delta:=\inf_{k\geqslant1}|S^{n_k}U_{\epsilon_k}|>0\fs
\]
}
\end{equation}

\begin{enumerate}[1.]
\item
If $\Gamma(\vartheta)+\Lambda(\vartheta)>0$, then
\begin{equation}\label{eq:sigma_+}
{\sigma}_+(\vartheta)
=
\frac{\Gamma(\vartheta)+\Lambda(\vartheta)}{\Lambda(\vartheta)}\cdot
{s_*}
\text{\quad and \quad}
\sigma_-(\vartheta)=0\fs
\end{equation}
\item If $\Gamma(\vartheta)+\Lambda(\vartheta)<0$, then
\begin{equation}\label{eq:sigma_-}
\sigma_-(\vartheta)=
-\frac{\Gamma(\vartheta)+\Lambda(\vartheta)}{\Lambda(\vartheta)}
\text{\quad and \quad}
\sigma_+(\vartheta)=0\fs
\end{equation}
\end{enumerate}

\end{theorem}

\begin{remark}[On the notion of regularity of a point $\vartheta$]\label{remark:regularity}\
\begin{enumerate}[a)]
\item
The set of points $\vartheta\in\T$ for which (\ref{eq:Lambda-Gamma-ass}) is violated has measure zero for each $S$-invariant measure by Birkhoff's Ergodic Theorem. Those points for for which (\ref{eq:extra-ass}) is violated have measure zero for each $S$-invariant Gibbs measure. Indeed, in section~\ref{subsec:regularity} we prove the stronger fact that the same is true for each $S$-invariant measure $\mu$ with the property that 
\begin{equation}\label{eq:weaker-than-Gibbs}
\mu(W_\epsilon)=\mathcal{O}\left(\left(\log\log\frac{1}{\epsilon}\right)^{-(1+q)}\right)
\quad\text{as $\epsilon\to0$ for some $q>0$}
\end{equation}
where $W_\epsilon$ is the $\epsilon$-neighbourhood of the set of endpoints of monotonicity intervals of $S$. (Observe that for each Gibbs measure $\mu$ there exists $t\in(0,1)$ such that $\mu(W_\epsilon)=\mathcal{O}(\epsilon^t)$, because $S$ is piecewise uniformly expanding.)
\item
If $S$ is an expanding $C^{1+}$-map of $\T$, then there is, for each $n\geqslant1$, a symmetric interval $U\in\mathcal{U}_n(v)$ with $|S^nU|=1$. Therefore (\ref{eq:extra-ass}) is satisfied for all $v\in\T$ in this case.
\item
If one replaces the symmetric intervals in the definition  of $\Sigma_\epsilon(v)$ by maximal monotonicity intervals, then (\ref{eq:extra-ass}) is satisfied for all Markov maps.

\end{enumerate}
\end{remark}

\begin{remark}
Numerical investigations related to equations (\ref{eq:sigma_+}) and (\ref{eq:sigma_-}) are presented in \cite{KJR2012}.
\end{remark}

\begin{remark}
In \cite{KeOt2012} we characterize the Hausdorff and packing dimension of the set $\{\theta\in\Theta: \hppl(\theta)=0\}$ and related ones using thermodynamic formalism for the map $S$. In other words, we study the local scaling behaviour of the set of zeros of $\hppl$. Theorems 1 - 3 extend this point of view in that they describe the local scaling behaviour of the subgraph of $\hppl$ in regions where $\hppl$ assumes values very close to zero.
\end{remark}

\subsection{The Anosov case}
In Example~\ref{ex:anosov} we described how Anosov surface diffeomorphisms driving a Hölder function $\hat g:\mathbb{T}^2\to(0,\infty)$ fit the general framework of this note. The basic observation is 
Proposition~\ref{prop:anosov-lemma} relating
the invariant graph $\hppl$ defined in (\ref{eq:hat_phi_def}) to its "one-sided" approximation $\ppl\circ\Pi$ which is the invariant graph for the system where the multiplier function $\hat{g}$ is replaced by $g\circ\Pi$.

Using Proposition~\ref{prop:anosov-lemma} and standard facts about Anosov surface diffeomorphisms, in particular that the stable and the unstable foliation are uniformly transversal and $C^{1+}$ \cite[Theorem III.3.1]{Mane1987}, 
one can deduce the following theorem from the results of the previous two subsections.

Recall from Example~\ref{ex:anosov} that $\hat{\mu}_{\text{SRB}}^-$ is the Sinai-Ruelle-Bowen measure of $\hS^{-1}$ and denote by $\hat\psi(t)$ the topological pressure of 
$-\log|D_u\hS^{-1}|-t\log\hat{g}$ under $\hS^{-1}$. As $\log\hat{g}$ is cohomologous to $\log g\circ\Pi$ by (\ref{eq:log_g-cohom}) and 
$\log|D_u\hS^{-1}|$ to $\log|S'|\circ\Pi$ by (\ref{eq:log_DS-cohom}), we have
\begin{equation}
\hat{\psi}(s)=\psi(s)\quad\text{and}\quad \hat{\psi}'(0)=-\hat{\mu}_{\text{SRB}}^-(\log\hat{g})
=-\mac(\log g)<0
\end{equation}
so that the zero ${s_*}>0$ of $\psi$ defined in section~\ref{subsec:global-sc} 
is at the same time the unique positive zero of $\hat{\psi}$.
\begin{theorem}
Let $\Theta=\mathbb{T}^2$ and let $\hS:\mathbb{T}^2\to\mathbb{T}^2$ be a $C^2$ Anosov diffeomorphism.  Suppose that $g:\mathbb{T}^2\to(0,\infty)$ is Hölder continuous. Then
the invariant graph $\hppl$ has the following properties:
\begin{enumerate}[1.]
\item
\begin{equation}\label{eq:anosov-th-1}
\lim_{\epsilon\to0}\frac{\log m^2\{\hppl<\epsilon\}}{\log\epsilon}={s_*}\fs
\end{equation}
\item
\begin{equation}\label{eq:anosov-th-2}
\lim_{\epsilon\to0}\frac{\log \hat\Xi_\epsilon}{\log\epsilon}=0
\text{\quad and \quad}
\lim_{\epsilon\to0}\frac{\log(1- \hat\Xi_\epsilon)}{\log\epsilon}={s_*}
\end{equation}
where
$\hat\Xi_\epsilon:=\frac{1}{\epsilon}\int_{\mathbb{T}^2}\min\{\hppl,\epsilon\}\,dm^2$ $(\epsilon>0)$,
so that
$1-\hat\Xi_\epsilon=\frac{1}{\epsilon}\int_{\mathbb{T}^2}(\epsilon-\hppl)^+\,dm^2$.
\end{enumerate}
\ \\[-2mm]
Furthermore, there is a measurable subset $\Theta_0\subseteq\Theta$, which has measure zero for each Gibbs measure of $T$, such that for each $\theta\in\Theta\setminus\Theta_0$ the limits
\begin{equation}\label{eq:Anosov-limits}
\hat{\Gamma}(\theta):=\lim_{n\to\infty}\frac{1}{n}\log\hat{g}_n(\theta)\quad\text{and}\quad \hat{\Lambda}(\theta):=\lim_{n\to\infty}\frac{1}{n}\log|D_u\hS^{-n}(\theta)|
\end{equation}
exist and satisfy
$\hat{\Gamma}(\theta)=\Gamma(\Pi\theta)$ and $\hat{\Lambda}(\theta)=\Lambda(\Pi\theta)$, and the following holds:
\begin{enumerate}[1.]
\setcounter{enumi}{2}
\item
If $\hat\Gamma(\theta)+\hat\Lambda(\theta)>0$, then
\begin{equation}\label{eq:anosov-lemma-3}
\lim_{\epsilon\to0}\frac{\log(1-\Sigma_\epsilon(\theta))}{\log\epsilon}
=
\frac{\hat\Gamma(\theta)+\hat\Lambda(\theta)}{\hat\Lambda(\theta)}\cdot
{s_*}
\text{\quad and \quad}
\lim_{\epsilon\to0}\frac{\log\Sigma_\epsilon(\theta)}{\log\epsilon}=0
\end{equation}
where 
$\Sigma_\epsilon(\theta):=
\frac1{\epsilon\cdot|U_\epsilon(\theta)|}\int_{U_\epsilon(\theta)}
\min\{\hppl,\epsilon\}\,dm^2$ and $U_\epsilon(\theta)$ is a $\epsilon$-neighbourhood of $\theta$ in $\mathbb{T}^2$, so that
$1-\Sigma_\epsilon(\theta)=
\frac1{\epsilon\cdot|U_\epsilon(\theta)|}\int_{U_\epsilon(\theta)}
(\epsilon-\hppl)^+\,dm^2$.
\item 
If $\hat\Gamma(\theta)+\hat\Lambda(\theta)<0$, then
\begin{equation}\label{eq:anosov-lemma-4}
\lim_{\epsilon\to0}\frac{\log\Sigma_\epsilon(\theta)}{\log\epsilon}
=
-\frac{\hat\Gamma(\theta)+\hat\Lambda(\theta)}{\hat\Lambda(\theta)}
\text{\quad and \quad}
\lim_{\epsilon\to0}\frac{\log(1-\Sigma_\epsilon(\theta))}{\log\epsilon}=0\fs
\end{equation}

\end{enumerate}

\end{theorem}

\begin{proof}
The existence of the limits in (\ref{eq:Anosov-limits}) is again a consequence of Birkhoff's theorem.
The identities $\hat{\Gamma}(\theta)=\Gamma(\Pi\theta)$ and $\hat{\Lambda}(\theta)=\Lambda(\Pi\theta)$ follow from the fact that $\log\hat{g}$ is cohomologous to $\log g\circ\Pi$ and $\log|D_u\hS^{-1}|$ to $\log|S'|\circ\Pi$, see the discussion before the theorem. 
In view of Remark~\ref{remark:regularity}a we can choose $\Theta_0$ such that all points in $\Theta\setminus\Theta_0$ are regular in the sense of Theorem~\ref{theo:main}. Then
all other claims follow from Theorems~\ref{theo:tail}~-~\ref{theo:main} along the following lines:
Let $c=\frac{a}{h(a)}e^{2\|\hat{b}\|_\infty}$. Then
\begin{equation}
\Pi^{-1}\{\ppl<c^{-1}\epsilon\}
\subseteq
\{\hppl<\epsilon\}
\subseteq
\Pi^{-1}\{\ppl<c\epsilon\}
\end{equation}
and
\begin{equation}
c\,(c^{-1}\epsilon-\ppl\circ\Pi)^+
\leqslant
(\epsilon-\hppl)^+
\leqslant
c^{-1}(c\epsilon-\ppl\circ\Pi)^+
\end{equation}
because of Proposition~\ref{prop:anosov-lemma}. Therefore it suffices to prove 
(\ref{eq:anosov-th-1}) and (\ref{eq:anosov-th-2}) for the graph $\ppl\circ\Pi$ instead of $\hppl$. As $\ppl\circ\Pi$ is constant along local stable manifolds, and as the passage to local coordinates is absolutely continuous with bounded Jacobian determinant (see  \cite[Proposition 4.2]{Bressaud2002} for details), there is a constant $C>0$ such that
\begin{equation}
C^{-1}
\leqslant
\frac{m^2\{\ppl\circ\Pi<\epsilon\}}{m\{\ppl<\epsilon\}}\ ,\;
\frac{\int_{\mathbb{T}^2}(\epsilon-\ppl\circ\Pi)^+\,dm^2}{\int_{\mathbb{T}^1}(\epsilon-\ppl)^+\,dm}
\leqslant
C\fs
\end{equation}
Now (\ref{eq:anosov-th-1}) and (\ref{eq:anosov-th-2}) follow from Theorem~\ref{theo:tail} and~\ref{theo:global-scaling}, respectively.
With essentially the same arguments, (\ref{eq:anosov-lemma-3}) and (\ref{eq:anosov-lemma-4}) both follow from Theorem~\ref{theo:main}.
\end{proof}

\section{Distortion estimates}

\subsection{Branch distortion}

Recall that $\mathcal{U}_n(\vartheta)$ denotes the family of all interval neighbourhoods $U$ of $\vartheta\in\T$ such that $S^n_{|U}:U\to S^nU$ is a diffeomorphism.

The following proposition is most important for estimating distortions along single branches $F_\theta^n:I\to I$. It uses only the concavity of $h:I\to I$. As
\begin{equation}
0<c_h:=\min\{a,h'(a)\}\leqslant\min\{h'(x): x\in I\}\co
\end{equation}
there is a constant $a_h>0$ such that 
\begin{equation}
h'(x)\geqslant e^{-a_hx}\text{ for all }x\in I\fs
\end{equation}

We will also use the following notation: 
For $n\geqslant1$ and $\vartheta\in\T$ define $f_{n,\vartheta}:\T\to\T$ by $f_{1,\vartheta}(x)=g(S\vartheta)h(x)$ and $f_{n,\vartheta}(x)=f_{1,\vartheta}(f_{n-1,S\vartheta}(x))$ if $n>1$. Observe that $f_{n,\vartheta}(x)=f_{n-1,\vartheta}(f_{1,S^{n-1}\vartheta}(x))$. By definition, $f_{n,\vartheta}(x)$ is always to be interpreted as a point in the fibre over $\vartheta$.
Observe also that $F^n_{\hS^{-n}\theta}(x)=f_{n,\Pi\theta}(x)$ for all $n\geqslant1$ and $\theta\in\Theta$.

For fixed $n\in\N$ and $x\in I$ let
\begin{equation}
x_{-i}=f_{n-i,S^i\vartheta}(x)\; (i=1,\dots,n)\co
\end{equation}
and observe that $x_{-n}=x$ is a point in the fibre over $S^n\vartheta$, i.e. at time $-n$, while $x_0$ is a point in the fibre over $\vartheta$, i.e. at time $0$.
Note that we suppress the $n$-dependence of $x_i$ in this notation.

For a given sequence $(\alpha_i)_{i\geqslant1}$ of positive real numbers let $A_n=\sum_{i=1}^n\alpha_i$ and set $C_n=c_h^{-1}\,e^{a_hA_n}$. If the sequence is summable we extend this notation to $A_\infty=\sum_{i=1}^\infty\alpha_i$ and $C_\infty=c_h^{-1}\,e^{a_hA_\infty}$.

\begin{proposition}\label{prop:distortion-1}
Let $(\alpha_i)_{i\geqslant1}$ be a sequence of positive real numbers.
For all $n\in\N$, $\vartheta\in\T$ and $x\in I$,
\begin{equation}\label{eq:main-dist-1}
\exp\left(-a_h\,\sum_{i=1}^nx_{-i}\right)
\leqslant
\frac{f_{n,\vartheta}'(x)}{f_{n,\vartheta}'(0)}\leqslant 1\co
\end{equation}
and if 
\begin{equation}\label{eq:main-dist-ass-1}
x_0=f_{n,\vartheta}(x)\leqslant C_n^{-1}\,\alpha_i\,g_i(\vartheta)\;(i=1,\dots,n)\co
\end{equation} 
then
\begin{equation}\label{eq:main-dist-2}
\sum_{i=1}^nx_{-i}
\leqslant
C_n\,x_0\,\sum_{i=1}^ng_i(\vartheta)^{-1}
\leqslant
 A_n\fs
\end{equation}
\end{proposition}

\begin{proof}
The second inequality of (\ref{eq:main-dist-1}) is an immediate consequence of the concavity of the branches. The first one follows from
\begin{align}\label{eq:as-in}
f_{n,\vartheta}'(x)
&=
f_{1,\vartheta}'(x_{-1})\cdot f_{n-1,S\vartheta}'(x_{-n})
=\;\dots\;
=
\prod_{i=0}^{n-1} f_{1,S^{i}\vartheta}'(x_{-i-1})\\
&=
\prod_{i=1}^{n} g(S^{i}\vartheta)\cdot\prod_{i=1}^{n} h'(x_{-i})
\geqslant
f_{n,\vartheta}'(0)\cdot\exp\left(-a_h\sum_{i=1}^n x_{-i}\right)\fs\nn
\end{align}
In order to prove (\ref{eq:main-dist-2}), it suffices to show that
\begin{equation}\label{eq:to-prove-new}
x_{-i}\leqslant C_n\,x_0\,g_i(\vartheta)^{-1}\leqslant\alpha_i\;(i=1,\dots,n)\fs
\end{equation}
As the second inequality is just a reformulation of (\ref{eq:main-dist-ass-1}), it remains to prove the first one.

For $i=1,\dots,n$ we have
\begin{equation*}
x_0=f_{i,\vartheta}(x_{-i})
\geqslant 
x_{-i}\cdot f_{i,\vartheta}'(x_{-i})
\end{equation*}
and, as in (\ref{eq:as-in}),
\begin{equation}
f_{i,\vartheta}'(x_{-i})
=
\prod_{j=1}^{i}g(S^{j}\vartheta)\cdot\prod_{j=1}^{i}h'(x_{-j})
\geqslant
g_i(\vartheta)\,h'(x_{-i})\cdot
\exp\left(-a_h\,\sum_{j=1}^{i-1}x_{-j}\right)\fs
\end{equation}
Hence
\begin{equation}\label{eq:xi-est-1}
x_{-i}\,h'(x_{-i})
\leqslant
x_0\,g_i(\vartheta)^{-1}\cdot
\exp\left(a_h\,\sum_{j=1}^{i-1}x_{-j}\right)\co
\end{equation}
and as $x_0\leqslant c_he^{-a_hA_n}\alpha_i g_i(\vartheta)$ for $i=1,\dots,n$ by assumption (\ref{eq:main-dist-ass-1}), it follows that 
\begin{equation}\label{eq:xi-est-2}
x_{-i}
\leqslant
\alpha_i\cdot\exp\left(-a_h A_n+a_h\sum_{j=1}^{i-1}x_{-j}\right)\fs
\end{equation}
For $i=1$ we see at once that $x_{-1}\leqslant\alpha_1\,e^{-a_hA_n}\leqslant\alpha_1$, and for $i=2,\dots,n$ it follows inductively that
\begin{equation}
x_{-i}
\leqslant
\alpha_i\cdot\exp
\left(-a_hA_n+a_h\sum_{j=1}^{i-1}\alpha_j\right)
\leqslant
\alpha_i\fs
\end{equation}
Combined with (\ref{eq:xi-est-1}) this yields
(\ref{eq:to-prove-new}), namely
\begin{equation}
x_{-i}
\leqslant
x_0\,g_i(\vartheta)^{-1}\,c_h^{-1}\,e^{a_hA_n}
=
C_n\,x_0\,g_i(\vartheta)^{-1}\fs
\end{equation}
\end{proof}

\begin{corollary}\label{coro:for-contradict}
Let $(\alpha_i)_{i\geqslant1}$ be as in the preceding proposition and suppose that $\alpha_i\leqslant1$ for all $i$. Then, for all $n\in\N$ and $\vartheta\in\T$, there exists $i\in\{1,\dots,n\}$ such that
\begin{equation}
\varphi_n(\vartheta)>C_\infty^{-1}\,\alpha_i\,g_i(\vartheta)\fs
\end{equation}
\end{corollary}
\begin{proof}
Suppose for a contradiction that there are $n\in\N$ and $\vartheta\in\T$ such that
\begin{equation}\label{eq:for-contadict1}
\varphi_n(\vartheta)\leqslant C_\infty^{-1}\,\alpha_i\,g_i(\vartheta)\quad (i=1,\dots,n)\fs
\end{equation}
Now Proposition~\ref{prop:distortion-1} implies 
\begin{equation}
\varphi_n(\vartheta)
=
f_{n,\vartheta}(a)
\geqslant
a\,f_{n,\vartheta}'(a)
\geqslant
a\,g_n(\vartheta)\,e^{-a_hA_n}
\geqslant
c_h\,g_n(\vartheta)\,e^{-a_hA_n}\,\alpha_n
=
C_n^{-1}\,g_n(\vartheta)\,\alpha_n
\end{equation}
which contradicts (\ref{eq:for-contadict1}) for $i=n$, because
$C_n<C_\infty$.
\end{proof}

\subsection{Area distortion}

Here are some consequences of the estimates from the previous section for "telescoping" certain small areas in $M=\T\times I$. 
Recall that $D$ is the distortion constant from Hypothesis~\ref{hypo:2} and Remark~\ref{remark:distortion}.
Denote also by $m^2$ the $2$-dimensional Lebesgue measure on $\T\times I$.
For $n\in \N$ and $U\in\mathcal{U}_n(\vartheta)$ we define the maps
\begin{equation}
f_{n,U}:S^n(U)\times I\to M, \quad(\vartheta,x)\mapsto(({S^n}_{|U})^{-1}(\vartheta),f_{n,\vartheta}(x))\fs
\end{equation}

\begin{proposition}\label{prop:distortion-2}
In the situation of Proposition \ref{prop:distortion-1}, let $(\alpha_i)_{i\geqslant1}$ be a \emph{summable} sequence. Then for all $\vartheta\in\T$, all $n\in\N$, all $U\in\mathcal{U}_n(\vartheta)$,
all
$H>0$ such that
\begin{equation}\label{eq:main-dist-ass}
\left|(S^n)'(\vartheta)\cdot g_i(\vartheta)\right|^{-1}
\leqslant
H^{-1}D^{-1}C_\infty^{-1}\,\alpha_i\text{ for }(i=1,\dots,n)\co
\end{equation}
and for $\tilde{\vartheta}\in U$ and $\tilde{x}\in I$ with
\begin{equation}\label{eq:ass-small-x0-2}
f_{n,\tilde{\vartheta}}(\tilde{x})
\leqslant 
H\cdot |(S^n)'({\vartheta})|^{-1}
\end{equation}
the following holds:
\begin{compactenum}[1.]
\item 
\begin{equation}\label{eq:main-dist-3}
e^{-a_hA_\infty}
\leqslant
\frac{f_{n,\tilde\vartheta}'(\tilde{x})}{f_{n,\tilde\vartheta}'(0)}\leqslant1\fs
\end{equation}
\item
\begin{equation}\label{eq:main-dist-3a}
D^{-1}e^{-a_hA_\infty}
\leqslant
\frac{f_{n,\tilde\vartheta}'(\tilde{x})}{f_{n,\vartheta}'(0)}
\leqslant
D\fs
\end{equation}
\item
For the Jacobian $Jf_{n,U}$,
\begin{equation}\label{eq:main-dist-4}
D^{-2}e^{-a_hA_\infty}
\leqslant
\frac{Jf_{n,U}(S^n\tilde{\vartheta},\tilde{x})}{Jf_{n,U}(S^n\vartheta,0)}
\leqslant
D^2\fs
\end{equation}
\item
For measurable $V,W\subseteq S^n(U)\times I$, 
\begin{equation}\label{eq:volume-dist}
\left(D^{4} e^{a_hA_\infty}\right)^{-1}
\leqslant
\frac{m^2(V)}{m^2(W)}
\bigg/
\frac{m^2({f_{n,U}}V)}{m^2({f_{n,u}}W)}
\leqslant
D^4e^{a_hA_\infty}\fs
\end{equation}
\end{compactenum}
\end{proposition}

\begin{proof}
1.\; This follows from Proposition \ref{prop:distortion-1} once we have checked that $f_{n,\tilde{\vartheta}}(\tilde{x})\leqslant C_\infty^{-1} \alpha_i g_i(\tilde{\vartheta})$ for $i=1,\dots,n$: By (\ref{eq:ass-small-x0-2}), (\ref{eq:main-dist-ass})  and 
Hypothesis~\ref{hypo:2},
\begin{equation}\label{eq:dist-intermed-1}
f_{n,\tilde{\vartheta}}(\tilde{x})
\leqslant
H\cdot |(S^n)'({\vartheta})|^{-1}
\leqslant
D^{-1} g_i(\vartheta) C_\infty^{-1}\alpha_i
\leqslant
g_i(\tilde\vartheta) C_\infty^{-1}\alpha_i\fs
\end{equation}
2.\; As 
\begin{displaymath}
\frac{f_{n,\tilde\vartheta}'(\tilde{x})}{f_{n,\vartheta}'(0)}
=
\frac{f_{n,\tilde\vartheta}'(\tilde{x})}{f_{n,\tilde\vartheta}'(0)}
\cdot
\frac{g_n(\tilde{\vartheta})}{g_n(\vartheta)}\co
\end{displaymath}
this follows at once from Hypothesis~\ref{hypo:2} and (\ref{eq:main-dist-3}).
\\[2mm]
3.\; 
Due to the skew product structure of $f_{n,U}$, its Jacobian is
\begin{equation}
Jf_{n,U}(S^n\vartheta,x)=|(S^n)'(\vartheta)|^{-1}\,f_{n,\vartheta}'(x)\fs
\end{equation}
Hence
\begin{align*}
\frac{Jf_{n,U}(S^n\tilde{\vartheta},\tilde{x})}{Jf_{n,U}(S^n \vartheta,0)}
=
\frac{|(S^n)'({\vartheta})|}{|(S^n)'(\tilde\vartheta)|}
\cdot
\frac{f_{n,\tilde{\vartheta}}'(\tilde{x})}{f_{n,{\vartheta}}'(0)}\co
\end{align*}
and (\ref{eq:main-dist-4}) follows at once from Remark~\ref{remark:distortion} and (\ref{eq:main-dist-3a}).
\\[2mm]
4.\; This is an immediate consequence of (\ref{eq:main-dist-4}).
\end{proof}

\section{The distribution of $\ppl$: Proofs}
\label{sec:proof-global-dist}

\subsection{Proof of Theorem \ref{theo:tail}}

The proof of Theorem \ref{theo:tail} is inspired by proofs of a related result in queuing theory, namely the determination of Loyne's exponent \cite{Loynes1962} for the stationary distribution of Lindley's recursion
\cite{Lindley1952}, see also \cite{Glynn1994} and in particular \cite[Lemmas 4 and 5]{Lelarge2008}.

Recall the weighted Perron-Frobenius operators $\LL_s$ defined in (\ref{eq:PF}) and the notation $\psi(s)=\log\rho(\LL_s)$. We noticed already that $\psi(0)=0$, $\psi'(0)<0$, and that there is a unique ${s_*}>0$ such that $\psi({s_*})=0$ and $\psi'({s_*})>0$.
For technical reasons we prove a slightly stronger statement than Theorem~\ref{theo:tail}, namely: For each family $(J_x)_{x>0}$ of subintervals of $\T$ with $\inf_{x>0}|J_x|>0$ we have
\begin{equation}\label{eq:tail-scaling-J}
\lim_{x\to\infty}\frac{1}{x}\log m\{\vartheta\in {J_x}: \log\ppl(\vartheta)<-x\}=-{s_*}\fs
\end{equation}

Fix any $s\in(0,{s_*})$ and choose $\delta>0$ such that $\rho(\LL_s)e^{3s\delta}<1$.
There is a constant $C>0$ that depends on $s$ and $\delta$ such that 
\begin{equation}
\|\LL_s^n1\|_1\leqslant C\left(\rho(\LL_s)e^{s\delta}\right)^n\leqslant C e^{-2ns\delta}\text{\quad for all }n\geqslant 1\fs
\end{equation}
For $\kappa>0$ denote
\begin{equation}
A_\kappa
=
\left\{\vartheta\in \T: \exists n\geqslant1\text{ such that }
g_n(\vartheta)\leqslant\kappa e^{n\delta}\right\}\fs
\end{equation}
\begin{lemma}\label{lemma:A_kappa}
There is a constant $C>0$ that depends on $t$ and $\delta$ such that for all $\kappa>0$
\begin{equation}\label{eq:Cramer}
m(A_\kappa)
\leqslant C\cdot\kappa^s\fs
\end{equation}
\end{lemma}

\begin{proof}
As $s>0$, we have the usual Cram\'er type estimate for each $n\geqslant1$:
\begin{align}
m\left\{\vartheta\in \T:
g_n(\vartheta)\leqslant\kappa e^{n\delta}\right\}
&=
m\left\{\vartheta\in \T:
\kappa^s e^{ns\delta}
\,
e^{-s\log g_n(\vartheta)}\geqslant1
\right\}
\nn\\
&\leqslant
\kappa^s e^{ns\delta}
\int_\T e^{-s\log g_n}\,dm\nn\\
&=
\kappa^s e^{ns\delta}
\int_\T\LL_0^n(e^{-s\log g_n})\,dm
=
\kappa^s e^{ns\delta}
\int_\T\LL_s^n(1)\,dm\nn\\
&\leqslant
C e^{-ns\delta}\cdot\kappa^s\fs
\end{align}
Summing this inequality over all $n=1,2,\dots$, we get 
(\ref{eq:Cramer}) with the constant $C/(e^{s\delta}-1)$, which depends again only on $\delta$ and $s$.
\end{proof}

We start the proof of (\ref{eq:tail-scaling-J}) with the upper estimate.
Let $\alpha_i=e^{-i\delta}$ $(i=1,2,\dots)$ so that $A_\infty=\sum_{i=1}^\infty\alpha_i=\frac{1}{e^\delta-1}$ and $C_\infty=c_h^{-1}e^{-a_hA_\infty}$ depend only on $\delta$.
Let $\vartheta\in\T\setminus A_\kappa$. Then $g_i(\vartheta)\alpha_i>\kappa$ for all $i\geqslant1$. Therefore, by Corollary~\ref{coro:for-contradict}, for all $n\in\N$ there exists $i\in\{1,\dots,n\}$ such that
\begin{equation}
\varphi_n(\vartheta)>C_\infty^{-1}\,\alpha_i\,g_i(\vartheta)
>C_\infty^{-1}\,\kappa
\end{equation}
and hence
\begin{equation}
\ppl(\vartheta)=\inf_{n\geqslant1}\varphi_n(\vartheta)
\geqslant
C_\infty^{-1}\,\kappa\fs
\end{equation}

Now fix $x>0$ and let $\kappa=e^{-x}C_\infty$. Then
$\ppl(\vartheta)\geqslant e^{-x}$ for $\vartheta\in\T\setminus A_\kappa$ so that, in view of Lemma~\ref{lemma:A_kappa},
\begin{equation}\label{eq:LD-upper-est}
\limsup_{x\to\infty}\frac{1}{x}\log 
m\{\vartheta\in {J_x}: \log\ppl(\vartheta)<-x\}
\leqslant
\limsup_{x\to\infty}\frac{1}{x}\log m(A_{e^{-x}C_\infty})=
-s\fs
\end{equation}
As this estimate applies to each $s\in(0,{s_*})$, this proves the upper estimate in (\ref{eq:tail-scaling-J}).

We turn to the lower estimate. As $\ppl(\vartheta)\leqslant\varphi_n(\vartheta)=f_{n,\vartheta}(a)
\leqslant a\,g_n(\vartheta)$ for all $n$ and all $\vartheta\in\T$, we have immediately that
\begin{equation}
m\{\vartheta\in {J_x}: \log\ppl(\vartheta)<-x\}
\geqslant
m\{\vartheta\in {J_x}: \log g_n(\vartheta)<-x-\log a\}
\end{equation}
for all $n\geqslant1$. 
Let $\alpha:=\psi'({s_*})>0$. Then,
for $n=\lceil\alpha^{-1}(x+\log a)\rceil$,
\begin{align}\label{eq:lower-LD}
\liminf_{x\to\infty}\frac{1}{x}\log m\{\vartheta\in {J_x}: \log\ppl(\vartheta)<-x\}
&\geqslant
\liminf_{n\to\infty}\frac{1}{\alpha n}\log\frac 
{m\{\vartheta\in {J_x}: -\log g_n(\vartheta)>n\alpha\}}{m(J_x)}\nn\\
&=
\frac{1}{\alpha}(\psi({s_*})-{s_*}\alpha)
=
\frac{\psi({s_*})}{\psi'({s_*})}-{s_*}
=-{s_*}\fs
\end{align}
This is a consequence of large deviations theory for the map $S$, details of which are provided in the appendix. Together with the upper estimate (\ref{eq:LD-upper-est}), it finishes the proof of Theorem \ref{theo:tail}.


\subsection{Proof of Theorem \ref{theo:global-scaling}}

Again we prove a "localized" version of this theorem: instead of the quantity $\Xi_\epsilon=\frac{1}{\epsilon}\int_\T\min\{\ppl(t),\epsilon\}\,dt$ we look at 
\begin{equation}
\Xi_\epsilon:=\frac{1}{\epsilon}\int_{J_\epsilon}\min\{\ppl(t),\epsilon\}\,dt
\end{equation}
for a family of intervals $J_\epsilon$ with $\inf_\epsilon|J_\epsilon|>0$.

We only have to show that
\begin{equation}\label{eq:t*_loc}
\lim_{\epsilon\to0}\frac{\log(1- \Xi_\epsilon)}{\log\epsilon}={s_*}>0\co
\end{equation}
because this implies at once that 
$\lim_{\epsilon\to0}\frac{\log \Xi_\epsilon}{\log\epsilon}=0$.
Recall that
\begin{equation}
1-\Xi_\epsilon=\frac{1}{\epsilon}
\int_{J_\epsilon}(\epsilon-\ppl(\vartheta))^+\,d\vartheta
\leqslant
m\{\vartheta\in J_\epsilon: \ppl(\vartheta)\leqslant\epsilon\}\fs
\end{equation}
Therefore we conclude from  (\ref{eq:tail-scaling-J}) that
\begin{equation}
\limsup_{\epsilon\to0}\frac{\log(1-\Xi_\epsilon)}{\log\epsilon}
\leqslant
\limsup_{\epsilon\to0}\frac{1}{\log\epsilon}
\log m\{\vartheta\in J_\epsilon: \ppl(\vartheta)\leqslant\epsilon\}
=
{s_*}\fs
\end{equation}
For the lower estimate observe that
\begin{equation}
1-\Xi_\epsilon=\frac{1}{\epsilon}
\int_{J_\epsilon}(\epsilon-\ppl(\vartheta))^+\,d\vartheta
\geqslant
\frac{1}{2}
m\left\{\vartheta\in J_\epsilon: \ppl(\vartheta)\leqslant{\epsilon}/{2}\right\}\fs
\end{equation}
This implies, by (\ref{eq:tail-scaling-J}) again,
\begin{equation}
\liminf_{\epsilon\to0}\frac{\log(1-\Xi_\epsilon)}{\log\epsilon}
\geqslant
\liminf_{\epsilon\to0}\frac{1}{\log(\epsilon/2)}
\log m\{\vartheta\in J_\epsilon: \ppl(\vartheta)\leqslant\epsilon/2\}
=
{s_*}\fs
\end{equation}

\section{The stability index: Proof of Theorem \ref{theo:main}}
\label{sec:proofs-local-scaling}

Let $U_{\epsilon_k}$ be a symmetric open interval neighbourhood of $\vartheta$ in $\mathcal{U}_{n_k}(\vartheta)$ satisfying the regularity assumption (\ref{eq:extra-ass}).
As 
$1\geqslant\int_{U_{\epsilon_k}}|(S^{n_k})'|\,dm=|S^{n_k}U_{\epsilon_k}|\geqslant\Delta$, it follows from Remark~\ref{remark:distortion} that
\begin{equation}\label{eq:S-eps}
\frac{\Delta}{2D}\leqslant\epsilon_k|(S^{n_k})'(\vartheta)|\leqslant\frac{D}{2}=H\fs
\end{equation}
Combining this with (\ref{eq:Lambda-Gamma-ass}) and (\ref{eq:extra-ass}) we obtain
\begin{equation}
\lim_{k\to\infty}\frac{\log\epsilon_{k+1}}{\log\epsilon_{k}}
=
\lim_{k\to\infty}\frac{\log|(S^{n_{k+1}})'(v)|}{\log|(S^{n_{k}})'(v)|}
=
\lim_{k\to\infty}\frac{n_{k+1}}{n_k}
=
1 \ .
\end{equation}
For each $\epsilon\in[\epsilon_{k+1},\epsilon_k]$ we have
\begin{equation}
\frac{\log\epsilon_{k+1}}{\log\epsilon_{k}}\frac{\log\Sigma_{\epsilon_{k+1}}(\vartheta)}{\log\epsilon_{k+1}}
\leqslant
\frac{\log\Sigma_\epsilon(\vartheta)}{\log\epsilon}
\leqslant
\frac{\log\epsilon_k}{\log\epsilon_{k+1}}\frac{\log\Sigma_{\epsilon_k}(\vartheta)}{\log\epsilon_k} \ ,
\end{equation}
and the same holds when $\Sigma_\epsilon(\vartheta)$ is replaced by $(1-\Sigma_\epsilon(\vartheta))$. Therefore it suffices to evaluate the limits for $\sigma_\pm(\vartheta)$ in (\ref{eq:local-scaling}) only along the sequence $(\epsilon_k)_{k\in\N}$.

\noindent 1.\;\emph{The case $\Gamma(\vartheta)+\Lambda(\vartheta)>0$:}
We check the assumptions of Proposition \ref{prop:distortion-2}: Let
\begin{equation}
\delta=\frac{1}{4}\min\left\{\Lambda(\vartheta),\Gamma(\vartheta)+\Lambda(\vartheta)\right\}
\end{equation}
and observe that $\delta>0$.
As $\vartheta$ is regular, there is a constant $C_\vartheta>0$
such that $g_k(\vartheta)>C_\vartheta
\,e^{k(\Gamma(\vartheta)-\delta)}$ and 
$|(S^k)'(\vartheta)|>C_\vartheta\,e^{k(\Lambda(\vartheta)-\delta)}$ for all $k\in\N$.
Set $\alpha_i=e^{-i\delta}$.
Then
\begin{align}
|(S^n)'(\vartheta)\cdot g_i(\vartheta)|^{-1}
&\leqslant
C_\vartheta^{-2}e^{-n(\Lambda(\vartheta)-\delta)-i(\Gamma(\vartheta)-\delta)}\nn\\
&\leqslant
C_\vartheta^{-2}e^{-(n-i)3\delta-i2\delta}
=
C_\vartheta^{-2}e^{-n\delta}\alpha_i e^{-2(n-i)\delta}\nn\\
&\leqslant
C_\vartheta^{-2}e^{-n\delta}\alpha_i
\end{align}
for all $n\in\N$ and all $i=1,\dots,n$.

Now fix the constant  $H$ from Proposition \ref{prop:distortion-2} as $H=\frac{D}{2}$, where $D$ is the basic distortion constant from
Hypothesis~\ref{hypo:2} and 
Remark~\ref{remark:distortion}.
Then, for all sufficiently large $n$, assumption (\ref{eq:main-dist-ass}) is satisfied
for all $i=1,\dots,n$. In particular,
\begin{equation}\label{eq:main-proof-1}
|(S^n)'(\vartheta)|^{-1}
\leqslant C_\vartheta^{-2}e^{-2n\delta}g_n(\vartheta)\fs
\end{equation}

Assume for a contradiction that
$f_{n_k,\tilde{\vartheta}}(a)\leqslant\epsilon_k$.
Then $f_{n_k,\tilde{\vartheta}}(a)\leqslant H|(S^{n_k})'(\vartheta)|^{-1}$ 
by (\ref{eq:S-eps}), so that also
(\ref{eq:ass-small-x0-2}) is satisfied, and (\ref{eq:main-dist-3a}) of Proposition~\ref{prop:distortion-2} yields
\begin{equation}
f_{n_k,\tilde{\vartheta}}(a)\geqslant a\,f'_{n_k,\tilde{\vartheta}}(a)
\geqslant 
aD^{-1}e^{-a_hA_\infty}g_{n_k}(v)
\geqslant
aD^{-1}e^{-a_hA_\infty}C_v^2e^{2n_k\delta}|(S^{n_k})'(\vartheta)|^{-1}
\end{equation}
which contradicts $f_{n_k,\tilde{\vartheta}}(a)\leqslant H|(S^{n_k})'(\vartheta)|^{-1}$
when $n_k$ is sufficiently large, say $n_k\geqslant N_0(v)$.

Therefore, $f_{n_k,\tilde{\vartheta}}(a)>\epsilon_k$ for all $n_k\geqslant N_0(\vartheta)$ and all $\tilde{\vartheta}\in U_{\epsilon_k}$, and there are functions $\delta_{n_k}:S^{n_k}U_{\epsilon_k}\to I$ such that 
$f_{n_k,\tilde{\vartheta}}(\delta_{n_k}(t))
={\epsilon_k}\leqslant H\,|(S^{n_k})'(\vartheta)|^{-1}$ for all $t\in S^{n_k}U_{\epsilon_k}$. 
As 
$\frac{\epsilon_k}{\delta_{n_k}(t)}=\frac{f_{n_k,\tilde{\vartheta}}(\delta_{n_k}(t))}{\delta_{n_k}(t)}
=f_{n_k,\tilde{\vartheta}}'(\tilde{x})$ for some $\tilde{x}=\tilde{x}(t)$,
we conclude from (\ref{eq:main-dist-3a})  that
\begin{equation}\label{eq:dist-intermed-2}
D^{-1}e^{-a_hA_\infty}
\leqslant
\frac{\epsilon_k}{\delta_{n_k}(t)g_{n_k}(\vartheta)}
\leqslant
D\quad\text{for all }t\in S^{n_k}U_{\epsilon_k}\fs
\end{equation}

In view of (\ref{eq:volume-dist}) we have
\begin{equation}
\left(D^4e^{a_hA_\infty}\right)^{-1}
\leqslant
\frac{\int_{S^{n_k}U_{\epsilon_k}}(\delta_{n_k}(t)-\ppl(t))^+\,dt}{\int_{S^{n_k}U_{\epsilon_k}}\delta_{n_k}(t)\,dt}
\bigg/
\frac{\int_{U_{\epsilon_k}}(\epsilon_k-\ppl(t))^+\,dt}{2\epsilon_k^2}
\leqslant
D^4e^{a_hA_\infty}\fs
\end{equation}
As the second quotient is just $1-\Sigma_{\epsilon_k}(\vartheta)$, this implies
\begin{equation}\label{eq:sigmaplus1}
{\sigma}_+(\vartheta)
=
\lim_{k\to\infty}\frac{\log(1-\Sigma_{\epsilon_k}(\vartheta))}{\log\epsilon_k}
=
\lim_{k\to\infty}\frac1{\log\epsilon_k}\cdot\log
\frac{\int_{S^{n_k}U_{\epsilon_k}}(\delta_{n_k}(t)-\ppl(t))^+\,dt}{\int_{S^{n_k}U_{\epsilon_k}}\delta_{n_k}(t)\,dt}
\end{equation}
provided the last limit exists.
Now let
\begin{displaymath}
\underline{\kappa}_{k}=D^{-1}\frac{\epsilon_k}{g_{n_k}(\vartheta)},\quad
\overline{\kappa}_{k}=De^{a_hA_\infty}\frac{\epsilon_k}{g_{n_k}(\vartheta)}
\end{displaymath} 
and observe that
\begin{equation}
\underline{\kappa}_k\leqslant\delta_{n_k}(t)\leqslant\overline{\kappa}_k
\quad\text{for all }t\in {S^{n_k}U_{\epsilon_k}}
\end{equation}
in view of (\ref{eq:dist-intermed-2}).
Therefore,
\begin{align}\label{eq:upper-lower-est}
D^{-2}e^{-a_hA_\infty}\,(1-\Xi_{\underline{\kappa}_{k}})
&=
\frac{\underline{\kappa}_{k}}{\overline\kappa_{k}}\,
(1-\Xi_{\underline\kappa_{k}})
\nonumber\\
&=
\frac{1}{\overline\kappa_{k}}
\int_{S^{n_k}U_{\epsilon_k}}\left(\underline\kappa_{k}-\ppl(t)\right)^+\,dt
\nonumber\\
&\leqslant
\frac{\int_{S^{n_k}U_{\epsilon_k}}(\delta_{n_k}(t)-\ppl(t))^+\,dt}{\int_{S^{n_k}U_{\epsilon_k}}\delta_{n_k}(t)\,dt}\\
&\leqslant
\Delta^{-1}
\frac{1}{\underline{\kappa}_{k}}
\int_{S^{n_k}U_{\epsilon_k}}\left(\overline{\kappa}_{k} -\ppl(t)\right)^+\,dt
\nonumber\\
&=
\Delta^{-1}
\frac{\overline{\kappa}_{k}}{\underline{\kappa}_{k}}\,(1-\Xi_{\overline{\kappa}_{k}})
\nonumber\\
&=
\Delta^{-1}D^2e^{a_hA_\infty}\,(1-\Xi_{\overline\kappa_{k}})\fs\nonumber
\end{align}
As, in view of (\ref{eq:Lambda-Gamma-ass}) and (\ref{eq:S-eps}),
\begin{equation}
\lim_{k\to\infty}\frac{\log\underline{\kappa}_k}{\log\epsilon_k}
=
\lim_{k\to\infty}\frac{\log\overline{\kappa}_k}{\log\epsilon_k}
=
1-\lim_{k\to\infty}\frac{\log g_{n_k}(v)}{\log\epsilon_k}
=
1+\frac{\Gamma(v)}{\Lambda(v)}
>
0\co
\end{equation}
and, observing (\ref{eq:t*_loc}),
\begin{equation}
\lim_{k\to\infty}\frac{\log(1-\Xi_{\underline{\kappa}_k})}{\log\underline{\kappa}_k}
=
{s_*}
=
\lim_{k\to\infty}\frac{\log(1-\Xi_{\overline{\kappa}_k})}{\log\overline{\kappa}_k}\co
\end{equation}
we conclude from (\ref{eq:sigmaplus1}) and (\ref{eq:upper-lower-est}) that
\begin{equation}
{\sigma}_+(\vartheta)
=
\lim_{k\to\infty}\left(
\frac{\log(1-\Xi_{\underline\kappa_{k}})}{\log\underline\kappa_{k}}\cdot\frac{\log\underline\kappa_{k}}{\log\epsilon_k}
\right)
=
{s_*}\cdot\frac{\Lambda(\vartheta)+\Gamma(\vartheta)}{\Lambda(\vartheta)}>0\fs\nonumber
\end{equation}
In particular, $\sigma_-(\vartheta)=0$.
\\[2mm]
2.\; \emph{The case $\Gamma(\vartheta)+\Lambda(\vartheta)<0$:}
In this case,
\begin{equation}
\lim_{k\to\infty}\frac{1}{n_k}\log \frac{g_{n_k}({\vartheta})}{\epsilon_{n_k}}
=
\Gamma(\vartheta)+
\Lambda(\vartheta)
<0\fs
\end{equation}
As the branches $f_{n_k,\tilde{\vartheta}}$ are concave, it follows at once 
that
\begin{equation}\label{eq:ppl-small}
\ppl(\tilde{\vartheta})
\leqslant\varphi_{n_k}(\tilde{\vartheta})= f_{n_k,\tilde{\vartheta}}(a)
\leqslant a g_{n_k}(\tilde{\vartheta})
<\epsilon_{k}e^{{n_k}(\Gamma(\vartheta)+\Lambda(\vartheta))/2}
\end{equation}  
uniformly for $\tilde{\vartheta}\in U_{\epsilon_k}$
when $n_k$ is sufficiently large.
This implies immediately that $\sigma_+(\vartheta)=0$.

In order to estimate of $\sigma_-(\vartheta)$ we will apply Proposition \ref{prop:distortion-1} directly. To this end we show that
\begin{equation}\label{eq:summable-1}
\sum_{i=1}^nf_{n-i,S^i\tilde{\vartheta}}(a)=o(n)\fs
\end{equation}  
Observe first that
\begin{equation}
f_{n-i,S^i\tilde{\vartheta}}(a)\leqslant a\,g_{n-i}(S^i\tilde\vartheta)
=
a\frac{g_n(\tilde{\vartheta})}{g_i(\tilde{\vartheta})}
\leqslant
a D^2\frac{g_n(\vartheta)}{g_i(\vartheta)}\fs
\end{equation}
Let 
\begin{equation}
\Delta(n):=\sup\left\{
\frac{1}{i}|\log g_i-i\Gamma(\vartheta)|: i\geqslant n\right\}\fs
\end{equation}
Then $\Delta(n)\to0$ as $n\to\infty$. Fix a monotone sequence $(j_n)_n$ of integers with $j_n\to\infty$ and $j_n/n\to0$, and define a second sequence $(\ell_n)_n$ as $\ell_n=\lfloor n\sqrt{\Delta(j_n)}\rfloor$. Then
\begin{align}
\sum_{i=1}^nf_{n-i,S^i\tilde{\vartheta}}(a)
&=
\sum_{i=1}^{j_n}f_{n-i,S^i\tilde{\vartheta}}(a)+
\sum_{i=j_n+1}^{n-\ell_n}f_{n-i,S^i\tilde{\vartheta}}(a)+
\sum_{i=n-\ell_n+1}^nf_{n-i,S^i\tilde{\vartheta}}(a)\nonumber\\
&\leqslant
(j_n+\ell_n)a+aD^2g_n(\vartheta)
\sum_{i=j_n+1}^{n-\ell_n}g_i(\vartheta)^{-1}\nonumber\\
&\leqslant
(j_n+\ell_n)a+aD^2e^{n(\Gamma(\vartheta)+\Delta(j_n))}
\sum_{i=j_n+1}^{n-\ell_n}e^{i(-\Gamma(\vartheta)+\Delta(j_n))}\nonumber\\
&\leqslant
(j_n+\ell_n)a+aD^2\frac{e^{-\Gamma(\vartheta)+\Delta(j_n)}}{e^{-\Gamma(\vartheta)+\Delta(j_n)}-1}e^{n(\Gamma(\vartheta)+\Delta(j_n))+(n-\ell_n)(-\Gamma(\vartheta)+\Delta(j_n))}\nonumber\\
&=
o(n)+O(e^{2n\Delta(j_n)+\ell_n(\Gamma(\vartheta)-\Delta(j_n))})\fs
\end{align}
As $\Gamma(\vartheta)<-\Lambda(\vartheta)<0$ and as 
$n\Delta(j_n)=o(\ell_n)$, the $O(.)$-expression is bounded in $n$. So (\ref{eq:summable-1}) is proved.

Now (\ref{eq:main-dist-1}) of Proposition \ref{prop:distortion-1}
shows that
\begin{equation}\label{eq:low-distortion}
e^{o(n_k)}
\leqslant
\frac{f_{n_k,\tilde{\vartheta}}'(x)}{g_{n_k}(\tilde{\vartheta})}
\leqslant
1
\end{equation}
uniformly for all $\tilde{\vartheta}\in U_{\epsilon_k}$ and all $x\in[0,a]$. In particular,
\begin{equation}\label{eq:dist-direct}
e^{o(n_k)}
\leqslant
\frac{f_{n_k,\tilde{\vartheta}}(a)}{g_{n_k}(\tilde{\vartheta})}
\leqslant
a
\end{equation}
uniformly for all
$\tilde{\vartheta}\in U_{\epsilon_k}$.

We turn to the determination of $\sigma_-(\vartheta)$. 
As in the proof of Proposition~\ref{prop:distortion-2} the distortion bound (\ref{eq:low-distortion}) implies analogous subexponential distortion bounds on the Jacobians $Jf_{n,U}$. Therefore, observing that 
$|{S^{n_k}U_{\epsilon_k}}|\geqslant\Delta>0$,
\begin{equation}\label{eq:also-nec}
e^{o(n_k)}
=
e^{o(n_k)}\,\frac{\int_{S^{n_k}U_{\epsilon_k}}\ppl(t)\,dt}{\int_{S^{n_k}U_{\epsilon_k}} a\,dt}
\leqslant
\frac{\int_{{U_{\epsilon_k}}}
\ppl(\tilde{\vartheta})\,d\tilde{\vartheta}}
{\int_{U_{\epsilon_k}}
\varphi_{n_k}(\tilde{\vartheta})\,d\tilde{\vartheta}}
\leqslant
1\fs
\end{equation}
As $\log\epsilon_k=-n_k\Lambda(\vartheta)+o(n_k)$ and 
$\varphi_{n_k}(\tilde{\vartheta})=f_{n_k,\tilde{\vartheta}}(a)
=g_{n_k}(\tilde{\vartheta})e^{o(n_k)}$ by (\ref{eq:dist-direct}), and as $\ppl(\vartheta)<\epsilon_k$ 
in view of (\ref{eq:ppl-small}), it follows from (\ref{eq:also-nec}) that
\begin{align}
{\sigma}_-(\vartheta)
&=
\lim_{k\to\infty}
\frac{\log\Sigma_{\epsilon_k}(\vartheta)}{\log\epsilon_k}
=
\lim_{k\to\infty}
\frac1{\log\epsilon_k}\left(
\log\frac{\int_{U_{\epsilon_k}}
\ppl(\tilde{\vartheta})\,d\tilde{\vartheta}}
{\epsilon_k|U_{\epsilon_k}|}\right)
\nn\\
&=
\lim_{k\to\infty}
\frac1{\log\epsilon_k}
\log\frac{\int_{U_{\epsilon_k}}
g_{n_k}(\tilde{\vartheta})\,d\tilde{\vartheta}}{2\epsilon_k^2}
=
\lim_{k\to\infty}
\frac{\log \epsilon_k^{-1}g_{n_k}(\vartheta)}{\log\epsilon_k}\nonumber\\
&=
-1+\frac{\Gamma(\vartheta)}{-\Lambda(\vartheta)}
=
-\frac{\Gamma(\vartheta)+\Lambda(\vartheta)}{\Lambda(\vartheta)}>0\fs
\end{align}

\section{Proofs for hyperbolic systems}
\label{sec:anosov-proofs}

\subsection{Proof of Proposition~\ref{prop:anosov-lemma}}
In the course of this proof we need the function $H(x):=\log\frac{h(x)}{x}$ which is well defined for $x\in(0,a]$ and which extends by continuity to $H(0):=0$. 
Note also that $H(a)\leqslant H(x)<0$ and 
$H'(x)=\frac{1}{h(x)}\left(h'(x)-\frac{h(x)}{x}\right)<0$ for $x\in(0,a]$.

From the definition of $F$ it follows that
\begin{equation}
\log F_{\hS^{-1}\theta}(x)=\log x+\log\hat{g}(\hS^{-1}\theta)+H(x)
\end{equation}
for $x\in(0,a]$
and, by induction,
\begin{equation}\label{eq:anosov-basic-id}
\log F^\ell_{\hS^{-\ell}\theta}(x)
=
\log x+\sum_{k=1}^\ell\log\hat{g}(\hS^{-k}\theta)
+\sum_{k=1}^\ell H(F_{\hS^{-\ell}\theta}^{\ell-k}(x))\fs
\end{equation}
Applied to $x=\hat{\varphi}_{n-\ell}(\hS^{-\ell}\theta)$ this yields
\begin{equation}\label{eq:two-sided}
\log\hat{\varphi}_n(\theta)
=
\log \hat{\varphi}_{n-\ell}(\hS^{-\ell}\theta)
+\sum_{k=1}^\ell\log\hat{g}(\hS^{-k}\theta)
+\sum_{k=1}^\ell H(\hat{\varphi}_{n-k}(\hS^{-k}\theta))\fs
\end{equation}
If we apply the same reasoning to the system with multiplier $g\circ\Pi$, we get
\begin{equation}\label{eq:one-sided}
\log\varphi_n(\Pi\theta)
=
\log \varphi_{n-\ell}(\Pi\hS^{-\ell}\theta))
+\sum_{k=1}^\ell\log{g}(\Pi\hS^{-k}\theta)
+\sum_{k=1}^\ell H(\varphi_{n-k}(\Pi\hS^{-k}\theta))\fs
\end{equation}
As $\log\hat{g}=\log g\circ\Pi+\hat{b}-\hat{b}\circ \hS^{-1}$
by (\ref{eq:log_g-cohom}), we can take the difference of (\ref{eq:two-sided}) and (\ref{eq:one-sided}) and obtain
\begin{equation}\label{eq:anosov-diff}
\begin{split}
\log\frac{\hat{\varphi}_n(\theta)}{\varphi_n(\Pi\theta)}
=
\log&\frac{\hat{\varphi}_{n-\ell}(\hS^{-\ell}\theta)}{\varphi_{n-\ell}(\Pi\hS^{-\ell}\theta)}
+\hat{b}(\hS^{-1}\theta)-\hat{b}(\hS^{-\ell-1}\theta)\\
&+\sum_{k=1}^\ell\left(H(\hat{\varphi}_{n-k}(\hS^{-k}\theta))-H(\varphi_{n-k}(\Pi\hS^{-k}\theta))\right)\fs
\end{split}
\end{equation}
Let
\begin{equation}
\ell(n)=\min\left\{k\in\{0,\dots,n\}: 
\hat{\varphi}_{n-k}(\hS^{-k}\theta)
\geqslant\varphi_{n-k}(\Pi\hS^{-k}\theta)\right\}\fs
\end{equation} 
The index $\ell(n)$ is well defined, because 
$\hat{\varphi}_{0}(\hS^{-n}\theta)=a=\varphi_{0}(\Pi\hS^{-n}\theta)$. 
We have $\hat{\varphi}_{n-k}(\hS^{-k}\theta)<\varphi_{n-k}(\Pi\hS^{-k}\theta)$
for $k=1,\dots,\ell(n)-1$, and as $H'<0$, we conclude from (\ref{eq:anosov-diff}) that
\begin{equation}
\begin{split}
\log\frac{\hat{\varphi}_n(\theta)}{\varphi_n(\Pi\theta)}
&\geqslant
-2\|\hat{b}\|_\infty
+\left(H(\hat{\varphi}_{n-\ell(n)}(\hS^{-\ell(n)}\theta))-H(\varphi_{n-\ell(n)}(\Pi\hS^{-\ell(n)}\theta))\right)\\
&\geqslant
-2\|\hat{b}\|_\infty+H(a)
\end{split}
\end{equation}
provided $\ell(n)\geqslant1$. If $\ell(n)0$, this estimate is trivially satisfied.
Similarly one proves that
$\log\frac{\hat{\varphi}_n(\theta)}{\varphi_n(\Pi\theta)}
\leqslant
2\|\hat{b}\|_\infty-H(a)$.
Therefore,
\begin{equation}
\left|\log\frac{\hat{\varphi}_n(\theta)}{\varphi_n(\Pi\theta)}\right|
\leqslant
2\|\hat{b}\|_\infty+|H(a)|\fs
\end{equation}
In the limit $n\to\infty$ we conclude that $\hppl(\theta)>0$ if and only if 
$\varphi_\infty(\Pi\theta)>0$ and that 
$|\log\hppl(\theta)-\log\varphi_\infty(\Pi\theta)|
\leqslant
2\|\hat{b}\|_\infty+|H(a)|$
for such $\theta$.

\subsection{The set of regular points and Remark~\ref{remark:regularity}a}
\label{subsec:regularity}
Recall that $W_\epsilon$ is the $\epsilon$-neighbourhood of the finite set $E$ of endpoints of monotonicity intervals of $S$ and that $\mu$ denotes some $S$-invariant probability measure. We assume that there is some $q>0$ such that
$\mu(W_\epsilon)=\mathcal{O}\left(\left(\log\log\frac{1}{\epsilon}\right)^{-(1+3q)}\right)$
as $\epsilon\to0$, which is equivalent to
(\ref{eq:weaker-than-Gibbs}).

Let $\tilde{n}_k:=\lfloor \exp(k^{\frac{1}{1+q}})\rfloor$, and observe that
$d_k:=\tilde{n}_{k+1}-\tilde{n}_k\geqslant C\ \exp(k^{\frac{1}{1+2q}})$ for some $C>0$. Fix $r>0$ and suppose that, for some $\vartheta\in\T$, $S^{n}\vartheta\in W_r$ for all $n\in(\tilde{n}_k,\tilde{n}_{k+1}]$. As $S$ is a piecewise expanding Markov map, $S(E)\subseteq E$ and $|(S^n)'|\geqslant C\lambda^n$ for some $C>0$ and $\lambda>1$. If $r>0$ is chosen sufficiently small, this implies that $S^{\tilde{n}_k}\vartheta\in W_{\lambda^{-d_k}}$.
Hence,
\begin{equation}
\begin{split}
&\hspace{-1cm}\mu\left\{\vartheta\in\T: \ S^n\vartheta\in W_r\text{ for all }
n\in(\tilde{n}_k,\tilde{n}_{k+1}]\right\}
\leqslant
\mu\left(S^{-\tilde{n}_k}W_{\lambda^{-d_k}}\right)
=
\mu\left(W_{\lambda^{-d_k}}\right)\\
&=
\mathcal{O}\left(\left(\log\log\lambda^{d_k}\right)^{-(1+3q)}\right)
=
\mathcal{O}\left(\left(\log d_k\right)^{-(1+3q)}\right)
=
\mathcal{O}\left(k^{-\frac{1+3q}{1+2q}}\right) \ .
\end{split}
\end{equation}
Now the Borel-Cantelli Lemma implies that for $\mu$-a.e. $\vartheta\in\T$ there is $k_\vartheta\in\N$ such that for all $k\geqslant k_\vartheta$ there is some $n_k\in(\tilde{n}_k,\tilde{n}_{k+1}]$ such that $S^{n_k}\vartheta\not\in W_r$. These $n_k$ satisfy
\begin{equation}
\limsup_{k\to\infty}\frac{n_{k+1}}{n_k}
\leqslant
\limsup_{k\to\infty}\frac{\tilde{n}_{k+2}}{\tilde{n}_k}
\leqslant
\limsup_{k\to\infty}\exp\left((k+2)^{\frac{1}{1+q}}-k^{\frac{1}{1+q}}\right)
=
1 \ ,
\end{equation}
and routine arguments for piecewise $C^{1+}$ expanding Markov maps show the existence of a constant $\Delta>0$ (depending on $r$ chosen above) such that (\ref{eq:extra-ass}) is satisfied.

\subsection{Anosov surface diffeomorphsims and their Markov maps}
\label{subsec:Anosov-Markov}
Choose one fixed $\hS^{-1}$-unstable fibre in each rectangle of the Markov partition $\{R_1,\dots,R_p\}$ and identify these $p$ fibres isometrically with intervals $J_1,\dots,J_p$. 
Denote by $J$ the disjoint union of $J_1,\dots,J_p$ and by $\varsigma:J\to\Theta=\mathbb{T}^2$ the map identifying the fibres and the intervals. Define
$\Pi:\Theta\to J$ as the map that projects a point $\theta\in R_i$ along its $\hS^{-1}$-stable fibre to the fibre $\varsigma(J_i)$ and then by $\varsigma^{-1}$ to $J_i$. Glueing the $J_i$ at their endpoints turns $J$ into a copy of $\T$ and affects only finitely many points in $J$.

Now we can define a map $S:\T\to \T$ by
$S(v)=\Pi(\hS^{-1}(\varsigma v))$. By construction, each $S(J_i)$ is a union of intervals $J_j$, and the resulting map is a Markov map w.r.t. the partition into intervals $J_i\cap S^{-1}J_j$. We must check that $S$ is piecewise $C^{1+}$.

Recall from \cite[eq. (8) in the proof of Lemma III.3.2]{Mane1987} that $\Pi$, the holonomy map along $\hS^{-1}$-stable fibres, is $C^{1+}$ with derivative 
\begin{equation}
D\Pi(\theta)
=
\lim_{N\to\infty}\left|\frac{D_u\hS^{-N}(\theta)}{D_u\hS^{-N}(\varsigma\Pi\theta)}\right|\fs
\end{equation}
Observe that $\varsigma\Pi\hS^{-1}\varsigma\Pi\theta=\varsigma\Pi\hS^{-1}\theta$ by construction of $\Pi$ and $\varsigma$. Therefore
\begin{equation}\label{eq:app-cohom}
\begin{split}
\frac{D\Pi(\theta)}{D\Pi(\hS^{-1}\theta)}
&=
\lim_{N\to\infty}
\left|\frac{D_u\hS^{-N}(\hS^{-1}\theta)D_u\hS^{-1}(\theta)}{D_u\hS^{-N}(\hS^{-1}\varsigma\Pi\theta)D_u\hS^{-1}(\varsigma\Pi\theta)}
\bigg/\frac{D_u\hS^{-N}(\hS^{-1}\theta)}
{D_u\hS^{-N}(\varsigma\Pi\hS^{-1}\theta)}\right|\\
&=
\left|\frac{D_u\hS^{-1}(\theta)}{D_u\hS^{-1}(\varsigma\Pi\theta)}\bigg/\lim_{N\to\infty}
\frac{D_u\hS^{-N}(\hS^{-1}\varsigma\Pi\theta)}{D_u\hS^{-N}(\varsigma\Pi\hS^{-1}\varsigma\Pi\theta)}\right|\\
&=
\left|\frac{D_u\hS^{-1}(\theta)}{D_u\hS^{-1}(\varsigma\Pi\theta)\cdot
D\Pi(\hS^{-1}\varsigma\Pi\theta)}\right|\\
&=
\left|\frac{D_u\hS^{-1}(\theta)}{S'(\Pi\theta)}\right|\fs
\end{split}
\end{equation}

\section{Large deviations for $S$}
Piecewise expanding mixing $C^{1+}$ Markov maps of $\T$ which are endowed with a positive $\alpha$-Hölder continuous weight function $g$ have the following property: There is some $\alpha'>0$ (that depends on $\alpha$ and the minimal expansion of $S$) such that the transfer operator $\LL_s$ introduced in (\ref{eq:PF}) has a simple leading eigenvalue $\lambda_s>0$ and
\begin{equation}
\LL_s^n\xi=\lambda_s^n \zeta_s m_s(\xi)+O(\gamma_s^n)
\end{equation}
for each function $\xi:\T\to\R$ that is $\alpha'$-Hölder restricted to each Markov interval of $S$. Here $\zeta_s$ is a strictly positive eigenfunction,
$m_s$ is a probability measure on $\T$ with full topological support, and $\gamma_s<\lambda_s$ \cite{Baladi2000,ParryPollicott}.

Suppose now that $(J_n)_{n\geqslant1}$ is a sequence of subintervals of $\T$ with $\inf_n|J_n|>0$. Fix $s\in\R$. Then $\inf_nm_s(J_n)>0$, 
because otherwise one could find a subsequence $(J_{n_i})$ with $\lim_{i\to\infty}m_s(J_{n_i})=0$ and a nontrivial interval $J$ that is contained in all these $J_{n_i}$. But then $m_s(J)=0$ in contradiction to the fact that $m_s$ has full support.
It follows that
\begin{equation}
\begin{split}
\lim_{n\to\infty}\frac{1}{n}\log\frac{\int_{J_n}e^{-s\log g_n}\,dm}{m(J_n)}
&=
\lim_{n\to\infty}\frac{1}{n}\log\int_\T\LL_s^n1_{J_n}\,dm\nn\\
&=
\lim_{n\to\infty}\frac{1}{n}\log\left(\lambda_s^nm_s({J_n})\int_\T\zeta_s\,dm+O(\gamma_s^n)\right)\\
&=
\log\lambda_s=\log\rho(\LL_s)=\psi(s)\co
\end{split}
\end{equation}
and this is a smooth strictly convex function of $s$. 
So we are in the situation to apply the large deviations theorem of Plachky/Steinebach \cite{Plachky1975}, and this yields the estimate in (\ref{eq:lower-LD}).

\bibliography{new-DFG}
\bibliographystyle{abbrv} 

\end{document}